\documentclass[onefignum,onetabnum]{siamart220329}

\usepackage{amsmath,enumerate,amssymb}
\usepackage{amsfonts,dsfont}
\usepackage{bbm}
\usepackage{mathrsfs}
\usepackage{color,microtype}
\definecolor{blueFont}{rgb}{0,0,1}

\usepackage{algorithm}
\usepackage{algorithmic}

\usepackage{graphicx}
\usepackage{subfigure}
\usepackage{cite}
\usepackage{bm}
\usepackage{booktabs}  

\def\dom{\mathop{\rm dom}}
\def\R{\mathbb{R}}
\def\epi{\mathop{\mathrm{epi}}}
\def\dist{\mathop{\mathrm{dist}}}
\def\Alg{{\rm proxCG}_{\mathds{1}\ell}^{\rm pen}}

\ifpdf
  \DeclareGraphicsExtensions{.eps,.pdf,.png,.jpg}
\else
  \DeclareGraphicsExtensions{.eps}
\fi

\newif\ifistoreview
\istoreviewtrue

\newcommand{\addColor}{\textcolor{blue}}
\newcommand{\add}[1]{\ifistoreview\addColor{#1}\else #1\fi}

\newsiamremark{remark}{Remark}
\newsiamremark{hypothesis}{Hypothesis}
\crefname{hypothesis}{Hypothesis}{Hypotheses}
\newsiamthm{claim}{Claim}

\newtheorem{example}{Example}[section]
\newtheorem{assumption}{Assumption}[section]

\headers{A single-loop prox-CG penalty method}{H. Zhang, L. Zeng and T. K. Pong}

\title{A single-loop proximal-conditional-gradient penalty method\thanks{Submitted to the editors \today.
\funding{Ting Kei Pong is supported partly by the Hong Kong Research Grants Council PolyU153004/23p. Liaoyuan Zeng was supported partly by the National Natural Science Foundation of China 12201389.
}}}

\author{Hao Zhang\thanks{The Hong Kong Polytechnic University, Hong Kong, People's Republic of China
  (\email{haaoo.zhang@connect.polyu.hk}).}
\and Liaoyuan Zeng\thanks{Zhejiang University of Technology, Hangzhou, People's Republic of China (\email{zengly@zjut.edu.cn}).}
\and   Ting Kei Pong\thanks{The Hong Kong Polytechnic University, Hong Kong, People's Republic of China
  (\email{tk.pong@polyu.edu.hk}).}
}

\usepackage{amsopn}

\def\argmin{\mathop{\rm arg\,min}}
\def\Argmin{\mathop{\rm Arg\,min}}

\allowdisplaybreaks[4]
\numberwithin{equation}{section}
\definecolor{blue}{rgb}{0,0,0}

\ifpdf
\hypersetup{
}
\fi

\begin{document}

\maketitle

\begin{abstract}
We consider the problem of minimizing a convex separable objective (as a separable sum of two proper closed convex functions $f$ and $g$) over a linear coupling constraint. We assume that $f$ can be decomposed as the sum of a smooth part having H\"older continuous gradient (with exponent $\mu\in (0,1]$) and a nonsmooth part that admits efficient proximal mapping computations, while $g$ can be decomposed as the sum of a smooth part having H\"older continuous gradient (with exponent $\nu\in (0,1]$) and a nonsmooth part that admits efficient linear oracles. Motivated by the recent \add{works \cite{argyriou2014hybrid,yurtsever2018conditional}}, we propose a \emph{single-loop} variant of the standard penalty method, which we call a single-loop proximal-conditional-gradient penalty method ($\Alg$), for this problem. In each iteration of $\Alg$, we successively perform one proximal-gradient step involving $f$ and one conditional-gradient step involving $g$ on the quadratic penalty function, followed by an update of the penalty parameter. We present explicit rules for updating the penalty parameter and the stepsize in the conditional-gradient step in each iteration. Under a standard constraint qualification and domain boundedness assumption, we show that the objective value deviations (from the optimal value) along the sequence generated decay in the order of $t^{-\min\{\mu,\nu,1/2\}}$ with the associated feasibility violations decaying in the order of $t^{-1/2}$. Moreover, if the nonsmooth parts are indicator functions and the extended objective (i.e., the sum of the convex separable objective and the indicator function of the linear constraint) is a Kurdyka-{\L}ojasiewicz function with exponent $\alpha\in [0,1)$, then the distances to the optimal solution set along the sequence generated by $\Alg$ decay asymptotically at a rate of $t^{-(1-\alpha)\min\{\mu,\nu,1/2\}}$. Finally, we illustrate numerically the behavior of  $\Alg$ on \add{solving low rank Hankel matrix completion problems}.
\end{abstract}



\section{Introduction}
In this paper, we consider the following optimization problem with a convex separable objective and linear coupling constraint:
\begin{equation}
    \begin{array}{cl} \label{problem}
        \min\limits_{x\in \mathcal{E}_1, y\in \mathcal{E}_2}    &f(x)+ g(y) \\
        {\rm{s.t.}} &Ax+By=c,
        \end{array}
\end{equation}
where ${\cal E}$, ${\cal E}_1$ and ${\cal E}_2$ are finite dimensional Hilbert spaces, $c\in {\cal E}$, $A: \mathcal{E}_1 \rightarrow \mathcal{E}$ and $B: \mathcal{E}_2 \rightarrow \mathcal{E}$ are linear maps,
$f: \mathcal{E}_1 \rightarrow (-\infty, \infty]$ and
$g: \mathcal{E}_2 \rightarrow (-\infty, \infty]$ are proper closed convex functions; we also assume that the solution set of \eqref{problem} is nonempty.\footnote{Please refer to section~\ref{section_notation and preliminaries} for notation.} Model problems of this form naturally arise in applications such as data science, machine learning and statistics (see, e.g., \cite{zhu2017augmented,gu2018admm,wang2019global,boyd2011distributed,martins2011augmented,lin2013design}).

For many practical instances of \eqref{problem}, the design of efficient algorithms relies heavily on the efficiency of proximal mapping computations. In particular, when $f$ and $g$ in \eqref{problem} can be decomposed as the sum of a smooth part having Lipschitz continuous gradient and a nonsmooth part that admits efficient proximal mapping computations, algorithms such as the alternating direction method of multipliers (ADMM) and its variants can be suitably applied to solve \eqref{problem} (see, e.g., \cite{douglas1956numerical,glowinski1975approximation,fazel2013hankel,wang2014bregman,li2016majorized,li2016schur,sun2015convergent,eckstein1992douglas}), where each iteration involves two subproblems related to proximal mapping computations. Here, we say that a proper closed convex function $h:{\cal E}\to (-\infty,\infty]$ admits efficient proximal mapping computations if for all $\gamma > 0$, the proximal mapping of $\gamma h$ at any $x\in {\cal E}$ defined as
\[
{\rm Prox}_{\gamma h}(x):= \argmin_{u\in {\cal E}} \frac1{2\gamma}\|u - x\|^2 + h(u)
\]
can be computed efficiently, where $\argmin$ denotes the unique minimizer. The proximal mapping of many proper closed convex functions can be computed efficiently (see, e.g., \cite{Combettes2011,beck2017first}), and is a ``building block" for a large class of first-order methods.

Besides the proximal mapping, another important ``building block" for first-order methods is the linear oracle \cite{frank1956algorithm,pmlr-v28-jaggi13}: for a proper closed convex function $h:{\cal E}\to (-\infty,\infty]$, given $v\in {\cal E}$, the linear oracle of $h$ returns an element of
\[
\Argmin_{u\in {\cal E}}\ \langle v,u\rangle + h(u),
\]
where $\Argmin$ denotes the set of minimizers.
It is known that for some $h$ that arise in practice, the linear oracles can be executed efficiently while the proximal mappings can be difficult to compute; see \cite[Section~4.1]{pmlr-v28-jaggi13} for examples in the case when $h$ is an indicator function of a compact convex set. In particular, for instances of \eqref{problem} that arise in applications such as matrix completion, it can happen that $f$ only admits efficient proximal mapping computations, while $g$ only admits efficient linear oracles.
{\em Can one design an algorithm for \eqref{problem} that allows the flexible use of both proximal mapping computations and linear oracles?}

Recent seminal works along this direction of research are \cite{yurtsever2018conditional,yurtsever2019conditional,silveti2020generalized,argyriou2014hybrid}.
\add{The work \cite{argyriou2014hybrid} studied \eqref{problem} when $f$ admits efficient proximal mapping computations, $g$ can be expressed as the sum of a function that admits efficient linear oracles and a smooth part that has H\"olderian continuous gradient with exponent $\nu \in (0,1]$, the mapping $A$ in \eqref{problem} is the negative identity map and $c=0$. Their algorithm is essentially based on the following penalty function, where $\beta > 0$:
\[
\widehat F_\beta(x,y) := f(x) + g(y) + \frac{\beta}{2}\|x - By\|^2.
\]
In each iteration of their algorithm, given $\beta_t > 0$, they apply {\em one} step of the proximal gradient algorithm (with respect to $x$) and {\em one} step of the conditional gradient algorithm (with respect to $y$) to $\widehat F_\beta$, and then $\beta_t$ is updated. In particular, unlike classical penalty methods, this algorithm does not involve any inner loops for solving subproblems: e.g., the penalty function $\widehat F_{\beta_t}$ is not minimized up to a prescribed tolerance in each iteration. In this sense, their algorithm is a {\em single-loop} algorithm. According to \cite[Corollary~4.5]{argyriou2014hybrid}, when $f$ is in additional Lipschitz continuous, by setting $\beta_0 > 0$ and $\beta_{t+1} = \beta_0\sqrt{t+2}$ and the stepsize in the conditional-gradient step to be $2/(t+2)$ for all $t \geq 0$, the sequence $\{ y^t\}$ generated by their algorithm satisfies
\[
    \left| f(By^t)+ g(y^t) - {\sf val} \right| = {\cal O}(t^{-\min\{\nu, 1/2\}}),
\]
where ${\sf val}$ is the optimal value of  \eqref{problem}. The subsequent work \cite{yurtsever2018conditional} studied \eqref{problem} under the same assumptions on $f$, $A$ and $c$ in \cite{argyriou2014hybrid}, but they assumed that $g$ can be decomposed as the sum of the indicator function of a compact convex set admitting efficient linear oracles\footnote{When $h$ is the indicator function of a compact convex set, we call the linear oracles of $h$ linear oracles of the set.} and a smooth part having Lipschitz continuous gradient. When $f$ is the indicator function of a closed convex set ${\cal C}$, under a standard constraint qualification, it was proved in \cite[Theorem~3.3]{yurtsever2018conditional} that when $\beta_t = \beta \sqrt{t+2}$ (for some $\beta > 0$) and the stepsize in the conditional-gradient step is chosen as $2/(t+2)$ for all $t\ge 0$, the sequence $\{(x^t,y^t)\}$ generated by their algorithm satisfies
\begin{equation}\label{complexity}
|g(y^t) - {\sf val}| = {\cal O}(1/\sqrt{t})\ \ {\rm and}\ \ {\rm dist}(By^t,{\cal C}) = {\cal O}(1/\sqrt{t}),
\end{equation}
where ${\sf val}$ is the optimal value of \eqref{problem}.
}

Later, based on a similar {\em single-loop} idea, in \cite{yurtsever2019conditional}, the authors developed a single-loop augmented-Lagrangian-based method for \eqref{problem} under the same assumptions on $g$, $A$ and $c$ as in \cite{yurtsever2018conditional}, but allowed $f$ that admits efficient proximal mapping computations and can be written as the (separable) sum of the indicator function of a closed convex set and a Lipschitz continuous function. Under a standard constraint qualification and suitable choices of parameters, complexity results on the objective value deviations and feasibility violations similar to \eqref{complexity} were established; see Theorem 3.1 and Section 3.3 of \cite{yurtsever2019conditional}. Most recently and independently, the authors in \cite{silveti2020generalized} considered another special instance of \eqref{problem} with $A$ being an injective negative partial identity map, $f$ admitting efficient proximal mapping computations, and $g$ being the sum of two parts: a nonsmooth part that is Lipschitz continuous on its domain and admits efficient linear oracles, and a smooth part with gradient being $(G,\zeta)$-smooth -- this notion generalizes the notion of H\"older continuity; see \cite[Definition~2.5]{silveti2020generalized}. Their algorithm is also single-loop and makes use of both penalty and augmented Lagrangian functions, and asymptotic convergence was established under suitable assumptions.

Here, \add{motivated by \cite{argyriou2014hybrid,yurtsever2018conditional}}, we propose a single-loop algorithm based on a penalty function (see \eqref{penalty2} below) for solving \eqref{problem} in a general setting. Specifically, our framework allows general linear maps $A$ and $B$, an $f$ that can be decomposed as the sum of a smooth part $f_1$ having H\"older continuous gradient (with exponent $\mu\in (0,1]$) and a nonsmooth part $f_2$ that admits efficient proximal mapping computations,\footnote{Recall that this means the proximal mapping of $\gamma f_2$ can be computed efficiently for all $\gamma > 0$.} and a $g$ that can be decomposed as the sum of a smooth part $g_1$ having H\"older continuous gradient (with exponent $\nu\in (0,1]$) and a nonsmooth part $g_2$ that admits efficient linear oracles. Each iteration of our algorithm involves one step of the proximal gradient algorithm (with respect to $x$) and one step of the conditional gradient algorithm (with respect to $y$) applied to the penalty function. These steps can be performed efficiently thanks to our assumptions on $f_2$ and $g_2$.

In this paper, we analyze the convergence properties of the aforementioned single-loop algorithm under a standard constraint qualification.
Our contributions are summarized as follows:
\begin{enumerate}[{\rm (i)}]
  \item Under a mild domain boundedness assumption, we establish bounds on the objective value deviations and feasibility violations along the sequence generated by our algorithm. Specifically, when the penalty parameter $\beta_t = \beta_0 (t+1)^{1-\min\{\mu,\nu,1/2\}}$ (for some $\beta_0 > 0$) and the stepsize $\alpha_t$ in the conditional-gradient step is chosen as $2/(t+2)$ for all $t\ge 0$, the $\{(x^t,y^t)\}$ generated by our algorithm satisfies
\begin{equation}\label{complexity22}
|f(x^t) + g(y^t) - {\sf val}| = {\cal O}(t^{-\min\{\mu,\nu,1/2\}})\ \ {\rm and}\ \ \|Ax^t + By^t - c\| = {\cal O}(t^{-1/2}),
\end{equation}
where ${\sf val}$ is the optimal value of \eqref{problem}. These bounds match the bounds \eqref{complexity} \add{from \cite{argyriou2014hybrid,yurtsever2018conditional}} asymptotically under their settings.
    We also study the effect of choosing $\beta_t = \beta_0 (t+1)^\delta$ (for a general $\delta\in (0,1)$) for all $t\ge 0$ in our analysis.
    \item We show that if each of $f_2$ and $g_2$ is the sum of a real-valued convex function and the indicator function of a compact convex set and, moreover, the extended objective of \eqref{problem} (i.e., the sum of the objective and the indicator function of constraint set) is a Kurdyka-{\L}ojasiewicz (KL) function with exponent $\alpha\in [0,1)$, then $\{\dist((x^t,y^t),{\cal S})\}$ decays asymptotically at a rate of $t^{-(1-\alpha)\min\{\mu,\nu,1/2\}}$, where $\{(x^t,y^t)\}$ is generated by our algorithm with $\beta_t = \beta_0 (t+1)^{1-\min\{\mu,\nu,1/2\}}$ (for some $\beta_0 > 0$) and $\alpha_t = 2/(t+2)$ for all $t\ge 0$, and ${\cal S}$ is the optimal solution set of \eqref{problem}. We also present an example to illustrate how the KL exponent of the extended objective can be derived based on the recent studies of error bounds for conic feasibility problems \cite{lindstrom2023error,lindstrom2024optimal}, and develop a rule to deduce such a KL exponent from that of the Lagrangian of \eqref{problem}.
\end{enumerate}

The rest of this paper is organized as follows. In section~\ref{section_notation and preliminaries}, we review some notation and preliminary materials. Our algorithm is presented in section~\ref{section_algorithm_framework}. The complexity results such as \eqref{complexity22} are established in section~\ref{section_convergence_rate}, and the local convergence rate of $\{\dist((x^t,y^t),{\cal S})\}$ and the KL exponent of the extended objective of \eqref{problem} are studied in section~\ref{section_KL property}. Finally, we illustrate our convergence rate results numerically in section~\ref{section_numerical experiments}.

\section{Notation and preliminaries} \label{section_notation and preliminaries}
In this paper, $\mathcal{E}$,  $\mathcal{E}_1$ and $\mathcal{E}_2$ are finite dimensional Hilbert spaces. With an abuse of notation, we let $\langle \cdot , \cdot \rangle$ denote the standard inner product and $\| \cdot \|$ denote the associated norm in the underlying Hilbert space. For a linear map $A$, we use $A^*$ to denote its adjoint, and $\lambda_{\max}(A^*A)$ to denote the maximum eigenvalue value of $A^*A$.
We use $\mathbb{R}^n$ \add{(resp, $ \mathbb{C}^n $)} to denote the $n$-dimensional \add{real (resp., complex)} Euclidean space,
and $\mathbb{R}^{m \times n}$ \add{(resp., $ \mathbb{C}^{m\times n}$)} to denote the set of $m\times n$ \add{real (resp., complex)} matrices. For an $x \in \mathbb{R}^n$, we use $\|x\|_p$ to denote the $\ell_p$ norm, where $p \in [1,\infty]$.

For an extended-real-valued function $h: \mathcal{E} \rightarrow [-\infty, \infty]$, let $\dom h =\{x \in \mathcal{E}:  h(x)< \infty \}$ be its domain.
We use $\epi h$ to denote its epigraph,  which is defined as
$$\epi h= \{(x,t) \in \mathcal{E} \times \mathbb{R} : h(x) \leq t\}.$$
We say that $h$ is proper if $ \dom h \neq\emptyset$ and $h$ never attains $-\infty$. A proper function is closed if its epigraph is closed.
For a proper closed convex function $h: \mathcal{E} \rightarrow [-\infty, \infty]$, we use $\partial h(x)$ to denote its subdifferential at $x\in {\cal E}$,
i.e.,
$$
\partial h(x) = \left\{\xi \in {\cal E}: h(y)- h(x) \geq\langle \xi, y-x \rangle \ \ \  \forall y \in {\cal E}\right\},
$$
and let $\dom \partial h = \{x\in {\cal E}:\; \partial h(x)\neq \emptyset\}$.

For a nonempty convex set ${\cal C} \subseteq \mathcal{E}$, we use $\delta_{{\cal C}}$ to denote the indicator function, which is defined as
\begin{align}
    \delta_{{\cal C}}(x) = \left\{
        \begin{array}{cc}
            0  & x\in {\cal C},  \nonumber \\
            \infty & x \notin {\cal C}.  \nonumber
        \end{array}
     \right.  \nonumber
\end{align}
We use ${\rm{ri}}\,{\cal C}$ to denote the relative interior of ${\cal C}$. For a point $x \in \mathcal{E}$, we use ${\rm{dist}}(x, {\cal C}):= {\rm{inf}}_{y \in {\cal C}}\|x - y\|$ to denote the distance from $x$ to ${\cal C}$. Finally, when ${\cal C}$ is nonempty closed and convex, we use $P_{{\cal C}}(x)$ to denote the unique projection of $x$ onto ${\cal C}$.

Next, we recall some important definitions that will be used in our convergence analysis. We start with the following standard constraint qualification \add{ for \eqref{problem}; see, e.g., \cite[Appendix~B]{fazel2013hankel}, \cite[Assumption~2]{li2016majorized} and \cite[Assumption~2]{li2016schur}}.
\begin{definition}
    We say that \textbf{CQ} holds for \eqref{problem} if $ c \in A\,\mathrm{ri}\left(\dom f\right) +
    B\,\mathrm{ri}\left(\dom g\right).  \label{CQ}$
\end{definition}
Based on this {\textbf{CQ}}, it is standard to establish the optimality condition for \eqref{problem}, which is stated in the following lemma.
\begin{lemma}[{{Optimality condition}}] \label{optimal}
    Consider \eqref{problem} and suppose that the {\textbf{CQ}} holds. Let $(x^*,y^*)$ satisfy $Ax^* + By^* = c$. Then the following statements are equivalent.
    \begin{enumerate}[{\rm (i)}]
      \item The point $(x^*, y^*)$ is a minimizer of \eqref{problem}.
      \item There exists a $\bar{\lambda} \in \mathcal{E}$ such that $0 \in \partial f(x^*) + A^*\bar{\lambda}$ and $0 \in \partial g(y^*) + B^*\bar{\lambda}$.
    \end{enumerate}
\end{lemma}
\begin{proof}
    The result follows immediately from the definition of {\textbf{CQ}} in Definition~\ref{CQ} and \cite[Theorem~28.2]{Rockafellar+1970}.
\end{proof}

We also recall the definitions of Kurdyka-{\L}ojasiewicz (KL) property and KL exponent.  These notions are instrumental in analyzing the convergence properties of many contemporary first-order methods; see, e.g., \cite{attouch2009convergence,bolte2014proximal,li2018calculus,attouch2010proximal,attouch2013convergence}.
\begin{definition}[{KL property and exponent}]
    We say that a proper closed convex function $h: \mathcal{E}\rightarrow (-\infty,\infty]$ satisfies the KL property
    at $\bar{x}\in \dom \partial h$ if there exist $r\in (0,\infty]$, a neighborhood $U$ of $\bar{x}$
    and a continuous concave function $\phi: [0,r) \rightarrow \mathbb{R}_+$ such that
\begin{enumerate}[{\rm (i)}]
   \item $\phi(0) = 0$, $\phi$ is continuously differentiable on $(0,r)$ and $\phi' >0$.
   \item For all $x\in U$ with $h(\bar{x}) < h(x) < h(\bar{x}) +r$, it holds that
                            \begin{align}
                                \phi'(h(x)-h(\bar{x}))\mathrm{dist}(0,\partial h(x)) \geq 1. \nonumber
                            \end{align}
\end{enumerate}
    If $h$ satisfies the KL property at $\bar x\in {\rm dom}\,\partial h$ and the $\phi(t)$ above can be chosen as $ \rho t^{1-\alpha}$ for some $\rho>0$ and $\alpha\in[0,1)$, then we
    say that $h$ satisfies the KL property with exponent $\alpha$ at $\bar{x}$.

    A proper closed convex function $h$ satisfying the KL property at every point in $\dom \partial h$ is called a KL function. A proper closed convex function $h$ satisfying the KL property with exponent $\alpha \in [0, 1)$ at every point in $\dom\partial h$ is called a KL function with exponent $\alpha$.
\end{definition}
KL functions abound in contemporary applications; e.g., a proper closed convex semi-algebraic function is a KL function with exponent $\alpha\in [0,1)$; see \cite[Corollary~16]{bolte2007clarke}.

Finally, we recall the Abel's summation formula involving vector inner products. We provide a short proof for the convenience of the readers.
\begin{lemma} \label{lemma_Abel}
  For two sequences
$\{a^t\}$ and $\{b^t\} \subset {\cal E}$, it holds that for all $k\ge 2$,
\begin{align}\label{Abel_formula_vec}
  \sum_{t = 1}^{k-1} \langle a^t -a^{t+1}, b^t \rangle = \langle a^1, b^1 \rangle -\langle a^{k}, b^{k-1}\rangle  + \sum_{t=1}^{k-2}\langle a^{t+1}, b^{t+1} - b^t \rangle.
\end{align}
\end{lemma}

\begin{proof}
    For each $t\ge 1$, we have $\langle a^t -a^{t+1}, b^t \rangle = \langle a^t, b^t \rangle -\langle a^{t+1}, b^{t+1}\rangle + \langle a^{t+1},b^{t+1}- b^t\rangle$. Summing both sides of this equality from $t = 1$ to $k-1$ gives
  \begin{align*}
    \sum_{t = 1}^{k-1} \langle a^t -a^{t+1}, b^t \rangle &=\langle a^1, b^1 \rangle -\langle a^{k}, b^{k}\rangle  + \sum_{t=1}^{k-1}\langle a^{t+1}, b^{t+1} - b^t \rangle\\
    &= \langle a^1, b^1 \rangle -\langle a^{k}, b^{k-1}\rangle  + \sum_{t=1}^{k-2}\langle a^{t+1}, b^{t+1} - b^t \rangle.
\end{align*}
\end{proof}

\section{Algorithmic framework}\label{section_algorithm_framework}

In this section, we present our algorithm for \eqref{problem} and prove some auxiliary lemmas for our convergence analysis in subsequent sections. Before describing our algorithm, we first present two additional structural assumptions on \eqref{problem}. The first one states that the domains of $f$ and $g$ are bounded, while the second one states that $f$ and $g$ can be written as the sum of a smooth part and a possibly nonsmooth part; we impose suitable continuity assumptions on the gradient of the former part, and assume the efficient solvability of some associated subproblems for the latter part.

\begin{assumption} \label{Assumption 1}
  In \eqref{problem}, the domains of $f$ and $g$ are bounded, i.e.,
      \begin{align} \label{Definition_of_diameter}
    D_f := \sup\limits_{x_1,x_2\in \dom f}\|x_1 -x_2\rVert < \infty \ \ {\rm and}\ \
    D_g := \sup\limits_{y_1,y_2\in \dom g}\|y_1 -y_2\rVert<\infty.
\end{align}
\end{assumption}
\begin{assumption} \label{Assumption 2}
    In \eqref{problem}, we have $f = f_1 +f_2$ and $g = g_1 +g_2$,
    where $f_1: \mathcal{E}_1 \rightarrow \R$ and $g_1: \mathcal{E}_2 \rightarrow \R$ are convex and smooth, $f_2: \mathcal{E}_1 \rightarrow (-\infty, \infty]$ and $g_2: \mathcal{E}_2 \rightarrow (-\infty, \infty]$ are proper, closed and convex, and satisfy the following properties:
    \begin{enumerate}[{\rm (i)}]
      \item The gradients $\nabla f_1$ and $\nabla g_1$ are H\"older continuous on $\dom f$ and $\dom g$ respectively. In particular, this implies the existence of $\mu \in (0,1]$, $\nu\in (0,1]$, $M_f\ge 0$ and $M_g\ge 0$ such that
    \begin{equation}\label{Holder_property}
    \begin{aligned}
        f_1(y) &\leq f_1(x) + \left\langle \nabla f_1(x), y-x \right\rangle +\frac{M_f}{\mu +1}
        \|y-x\rVert^{\mu +1}\quad \forall x,y\in \dom f,  \\
        g_1(y) & \leq g_1(x) + \left\langle \nabla g_1(x), y-x \right\rangle +\frac{M_g}{\nu+1}
        \|y-x\rVert^{\nu+1} \quad \forall x,y\in \dom g.
    \end{aligned}
    \end{equation}
    \item The unique minimizer of the following problem can be computed efficiently for every $\gamma > 0$ and $u\in \mathcal{E}_1$:
        \begin{equation}\label{proxf}
        \min_{x\in \mathcal{E}_1} \frac1{2\gamma}\|x - u\|^2 + f_2(x).
        \end{equation}
    \item For every $v\in \mathcal{E}_2$, a minimizer of the following problem exists and can be computed efficiently:
    \begin{equation}\label{LOg}
    \min_{y\in \mathcal{E}_2} \langle v,y\rangle + g_2(y).
    \end{equation}
    \end{enumerate}
\end{assumption}
\begin{remark}[Comments on Assumption~\ref{Assumption 2}]
\begin{enumerate}[{\rm (i)}]
  \item In \eqref{Holder_property}, the $\mu$ and $\nu$ can be taken as the H\"olderian exponents of $\nabla f_1$ and $\nabla g_1$, respectively, and $M_f$ and $M_g$ can be chosen as the H\"olderian constants of $\nabla f_1$ and $\nabla g_1$,  respectively.
    In particular, if $f_1$ (resp.,~$g_1$) has Lipschitz continuous gradient on  $\dom f$ (resp., $\dom g$), then $\mu$ (resp.,~$\nu$) in \eqref{Holder_property} can be set to 1. It is well known that many loss functions in signal processing and machine learning have Lipschitz or H\"older continuous gradients; see, e.g., \cite{yu2017robust,zhou2015ell_1}.
  \item The objective of \eqref{proxf} is strongly convex and thus \eqref{proxf} has a unique minimizer. This unique minimizer is known as the proximal mapping of $\gamma f_2$ at $u$, and can be obtained efficiently for a wide variety of $f_2$; see, e.g., \cite{Combettes2011,beck2017first}.
  \item The problem \eqref{LOg} is typically called the linear oracle (see, e.g., \cite{nesterov2018complexity,ghadimi2019conditional,harchaoui2015conditional}). The efficiency in solving these oracles is the key for the efficient implementation of the conditional gradient algorithm (see, e.g.,  \cite{freund2016new,pmlr-v28-jaggi13,frank1956algorithm}).
\end{enumerate}
\end{remark}

\add{We illustrate the versatility of our assumptions in the following examples. }
\begin{example}
\add{Consider the following  compressed sensing problem                                                                                                                                                                                                                                                        with (heavy-tailed) generalized Gaussian measurement noise:
    \begin{equation}\label{exam_prob1}
        \begin{array}{cl}
             \min\limits_{x \in \R^n} & \|x\|_1 \\
            {\rm{s.t.}} & \|Ax - b\|_p \leq \sigma,
        \end{array}
    \end{equation}
    where $p \in (1,2)$, $\sigma >0$, $b\in \R^m$, $n\ge m \ge 2$ and $A \in \R^{m \times n}$ has full row rank. Notice that the feasible region of \eqref{exam_prob1} is nonempty (indeed, it contains $A^{\dagger}b$) and hence the solution
    set is nonempty.}

\add{Let $\widehat x = A^\dagger b$. Notice that for any solution $x^*$ to \eqref{exam_prob1}, we have
$\|x^*\|_\infty\le \|x^*\|_1 \leq \|\widehat x\|_1 < \|\widehat x\|_1 + 1$. Thus, the solution set of \eqref{exam_prob1} is contained in the interior of the set $\{x\in \R^n: \|x\|_\infty \leq 1+\|\widehat x\|_1 \} $.
Therefore, we can reformulate \eqref{exam_prob1} as follows by introducing a new variable $y$:
\begin{equation}\label{exam_reform}
  \begin{array}{cl}
    \min\limits_{x,y} & \|x\|_1 \\
    {\rm{s.t.}} & \|y\|_p \leq \sigma, \quad \|x\|_\infty\le \|\widehat x\|_1+1, \quad Ax-y=b.
  \end{array}
\end{equation}
One can check that this is a special case of \eqref{problem} and that {\bf CQ}, Assumptions~\ref{Assumption 1} and \ref{Assumption 2} hold. Specifically, one can take $f(x) = \|x\|_1 + \delta_{\|\cdot\|_{\infty}\le \|\widehat{x}\|_1+1}(x)$ and $g(y) = \delta_{\|\cdot\|_p \le \sigma}(y)$.
Then in Assumption~\ref{Assumption 2}, we can set $f_1 = 0$, $f_2 = f$, $M_f =0$, $\mu = 1$, and $g_1 = 0$, $g_2 = g$, $M_g= 0$, $\nu = 1$, and we note that Assumption~\ref{Assumption 2}(ii) and (iii) hold; see, e.g., \cite[Example 2.2]{beck2009fast} and \cite[Section 5.1]{ito2023parameter} for discussions of the corresponding \eqref{proxf} and \eqref{LOg}.
In addition, notice that $\dom f  = \{x: \|x\|_{\infty}\le \|\widehat{x}\|_1+1 \}$ and $\dom g  = \{y: \|y\|_p \le \sigma\}$. Then, we have $D_f  = 2\sqrt{n}\left( \|\widehat{x}\|_1+1 \right)$ and $D_g  = 2\sigma$ in Assumption~\ref{Assumption 1}. Finally, we also note that {\bf CQ} holds for \eqref{exam_reform} because $b = AA^\dagger b - 0$ and $A^\dagger b\in {\rm ri}(\dom f)$ and $0\in {\rm ri}(\dom g)$.}
\end{example}

\begin{example}\label{hankel_exam}
   \add{ Consider the following Hankel matrix completion problem\footnote{\add{In this example, Example~\ref{KL_hankel} and section~\ref{section_numerical experiments}, we use bold face letters to denote vectors / matrices with complex entries. Recall that for any ${\bm x} \in \mathbb{C}^n$, $\|{\bm x}\|_1 :=\sum_{j=1}^n|{\bm x}_j| = \sum_{j=1}^n\sqrt{|{\rm Re}({\bm x}_j)|^2 + |{\rm Im}({\bm x}_j)|^2}$.} }
    \begin{equation}\label{exam_prob2}
    \begin{array}{cl}
        \min\limits_{{\bm x}\in \mathbb{C}^n} & \| \Pi_{\Omega} (w\circ ({\bm x} - { \bar{\bm x}}))\|_1 \\
         {\rm s.t.}&  \|{\cal H}({\bm x})\|_* \le \sigma,
    \end{array}
\end{equation}
where $\bar {\bm x}\in \mathbb{C}^n$, $\sigma > 0$, $\|\cdot\|_*$ denotes the nuclear norm (i.e., the sum of singular values),
\[
    {\cal H}({\bm x}) := \begin{bmatrix}
  {\bm x}_1 & {\bm x}_2 & \cdots & {\bm x}_q\\
  {\bm x}_2 & {\bm x}_3 & \cdots & {\bm x}_{q+1} \\
  \vdots & \vdots & \cdots & \vdots \\
  {\bm x}_m & {\bm x}_{m+1} & \cdots & {\bm x}_n
\end{bmatrix}\in \mathbb{C}^{m\times q}
\]
with $m = \lceil \frac{n}{2} \rceil$, $q = n-m+1$, and ${\bm x}_j$ being the $j$-th component of ${\bm x}$, $\Omega \subseteq \{1,\ldots,n\}$ is the index set of the observed entries,  $\Pi_{\Omega}: \mathbb{C}^n \rightarrow \mathbb{C}^n$
is the sampling operator defined by $\left[ \Pi_{\Omega} {\bm y} \right]_j = {\bm y}_j$ if $j \in \Omega$, and $\left[ \Pi_{\Omega} {\bm y} \right]_j = 0$ otherwise, $\circ$ denotes the Hadamard (entry-wise) product, and
$w \in \R^n$ has its $j$-th entry being the number of entries along the $j$-th anti-diagonal of ${\cal H}({\bm x})$.}

\add{Problem \eqref{exam_prob2} with $\|\cdot\|_*$ replaced by ${\rm rank}(\cdot)$ and $\|\cdot\|_1$ replaced by $\|\cdot\|_2$ arises in recovery problems where the observed data  exhibits Hankel structure; see e.g., \cite{cai2023structured}.
Here, we use the nuclear norm as a proxy for the rank function and attempt to reconstruct the original signal from its noise-corrupted (specifically, Laplacian noise), partial observations $\Pi_{\Omega}({\bar{\bm x}})$ via solving \eqref{exam_prob2}.}

\add{Notice that for every ${\bm x}$ satisfying $\|{\cal H}(\bm x)\|_* \le \sigma$,  it holds that $\| {\bm x}\|_2 \le \sigma <\sigma +1$; this implies that $\|{\bm x} - \Pi_\Omega(\bar {\bm x})\|_2\le \|{\bm x}\|_2 + \|\Pi_\Omega({\bar{\bm x}})\|_2 < \sigma + \|\Pi_\Omega({\bar{\bm x}})\|_2 + 1$.
Therefore,  we can reformulate \eqref{exam_prob2} as follows:
 \begin{equation}\label{reform_prob2}
    \begin{array}{cl}
        \min\limits_{{\bm x}\in \mathbb{C}^n, {\bm Y}\in \mathbb{C}^{m\times q}} & \| \Pi_{\Omega} (w\circ ({\bm x} - {\bar{\bm x}}))\|_1 \\
         {\rm s.t.}&  \|{\bm Y}\|_* \le \sigma,~\|{\bm x} - \Pi_\Omega(\bar {\bm x})\|_2 \le \sigma+\|\Pi_\Omega({\bar{\bm x}})\|_2+1,~{\bm Y} = {\cal H}(\bm x).
    \end{array}
\end{equation}
Writing ${\bm x} = x_{_{\cal R}} + i x_{_{\cal I}}$, $\bar{\bm x} = \bar{x}_{_{\cal R}} + i\bar{x}_{_{\cal I}}$ and ${\bm Y} = Y_{_{\cal R}} + iY_{_{\cal I}}$ where $x_{_{\cal R}}$, $x_{_{\cal I}}$, $\bar{x}_{_{\cal R}}$, $\bar{x}_{_{\cal I}} \in \mathbb{R}^n$ and $Y_{_{\cal R}}$, $Y_{_{\cal I}}\in \mathbb{R}^{m \times q}$ denote the real and imaginary parts of ${\bm x}$, $\bar {\bm x}$ and ${\bm Y}$, respectively, we see that \eqref{reform_prob2} is equivalent to the following problem:
 \begin{equation}\label{reform_prob3}
    \begin{array}{cl}
        \min\limits_{
        \begin{subarray}{c}
          x_{_{\cal R}},x_{_{\cal I}} \in \R^n\\
           Y_{_{\cal R}}, Y_{_{\cal I}} \in \R^{m\times q}
        \end{subarray}  } & \sum_{j\in \Omega} w_j\sqrt{(x_{_{\cal R}} -\bar{x}_{_{\cal R}})_j^2
        + (x_{_{\cal I}} - \bar{x}_{_{\cal I}})_j^2} \\
         {\rm s.t.}&  \|Y_{_{\cal R}} + i Y_{_{\cal I}}\|_* \le \sigma,~ \|x_{_{\cal R}} + ix_{_{\cal I}} - \Pi_\Omega(\bar {\bm x})\|_2 \le \sigma + \|\Pi_\Omega({\bar{\bm x}})\|_2 +1,\\
         & Y_{_{\cal R}}-{\cal H}(x_{_{\cal R}}) =0,~ Y_{_{\cal I}} - {\cal H}(x_{_{\cal I}}) = 0.
    \end{array}
\end{equation}
One can check that \eqref{reform_prob3} is another special case of \eqref{problem} and that {\bf CQ}, Assumptions~\ref{Assumption 1} and \ref{Assumption 2} hold. Indeed,
we can set $f(x_{_{\cal R}}, x_{_{\cal I}}) \!=\!  \sum_{j\in \Omega}\! w_j\sqrt{(x_{_{\cal R}} -\bar{x}_{_{\cal R}})_j^2
        \!+\! (x_{_{\cal I}} - \bar{x}_{_{\cal I}})_j^2} + \delta_{\|\cdot + i\cdot - \Pi_\Omega(\bar {\bm x})\|_2 \le \sigma + \|\Pi_\Omega({\bar{\bm x}})\|_2 +1}(x_{_{\cal R}}, x_{_{\cal I}}) $
and  $g(Y_{_{\cal R}},Y_{_{\cal I}}) = \delta_{\|\cdot + i \cdot \|_* \le \sigma}(Y_{_{\cal R}},Y_{_{\cal I}})$. Then in Assumption~\ref{Assumption 2}, we can take $f_1 = 0$, $f_2 = f$, $M_f =0$, $\mu = 1$, and $g_1 = 0$, $g_2 = g$, $M_g= 0$, $\nu = 1$, and we note that Assumption~\ref{Assumption 2}(ii) and (iii) hold; see also section~\ref{section_numerical experiments} for the computation of the corresponding \eqref{proxf} and \eqref{LOg}.
In addition, notice that $\dom f  = \{(x_{_{\cal R}}, x_{_{\cal I}}): \|x_{_{\cal R}} + ix_{_{\cal I}} - \Pi_\Omega(\bar {\bm x})\|_2 \le \sigma + \|\Pi_\Omega({\bar{\bm x}})\|_2 +1 \}$ and $\dom g  = \{(Y_{_{\cal R}},Y_{_{\cal I}}): \|Y_{_{\cal R}} + i Y_{_{\cal I}}\|_* \le \sigma\}$.
Then, we have $ D_f = 2(\sigma + \|\Pi_\Omega({\bar{\bm x}})\|_2 +1 )$ and $D_g = 2\sigma$ in Assumption~\ref{Assumption 1}. Finally, one can deduce from $(0,0)\in {\rm ri}(\dom f)$ and $(0,0)\in {\rm ri}(\dom g)$ that {\bf CQ} holds for \eqref{reform_prob3}.}
\end{example}

We now describe our algorithm for solving \eqref{problem} under Assumptions~\ref{Assumption 1} and \ref{Assumption 2}. Our algorithm takes advantage of the efficiency in solving \eqref{proxf} and \eqref{LOg}, and is inspired by the recent \add{works \cite{yu2020rc,argyriou2014hybrid,yurtsever2018conditional}}, which proposed single-loop penalty-based methods for special instances of \eqref{problem}. Our algorithm is also a single-loop penalty-based method, and is obtained by simplifying a standard penalty method for \eqref{problem}.

Specifically, in a standard implementation of the penalty method, in each iteration, we fix a penalty parameter $\beta_t$ and consider the following penalty function for \eqref{problem}:
\begin{align}\label{penalty2}
    \widetilde F_{\beta_t}(x,y) := f_1(x) + f_2(x) + g_1(y) + g_2(y) + \frac{\beta_t}{2}\|Ax+ By-c \rVert^2.
\end{align}
Notice that in view of Assumption~\ref{Assumption 2}, for each fixed $\beta_t$, one can approximately minimize $\widetilde F_{\beta_t}$ by an alternating minimization scheme, where one can apply the proximal gradient algorithm for approximately minimizing $\widetilde F_{\beta_t}$ with respect to $x$ (whose subproblems can be solved efficiently thanks to Assumption~\ref{Assumption 2}(ii)), and then apply the conditional gradient algorithm for approximately minimizing $\widetilde F_{\beta_t}$ with respect to $y$ (whose subproblems can be solved efficiently thanks to Assumption~\ref{Assumption 2}(iii)). Upon obtaining an approximate minimizer $(\widetilde x^t,\widetilde y^t)$ of $\widetilde F_{\beta_t}$, one then updates $\beta_t$ and minimizes $\widetilde F_{\beta_{t+1}}$, using $(\widetilde x^t,\widetilde y^t)$ as the initial point. While the above standard penalty method is natural, it involves multiple inner loops and can be inefficient in practice.

Here, following the ideas in the recent \add{works \cite{yurtsever2018conditional,yurtsever2019conditional,argyriou2014hybrid,silveti2020generalized,yu2020rc}}, we apply {\em one} step of the proximal gradient algorithm and {\em one} step of the conditional gradient algorithm for each fixed $\beta_t$ in the penalty method described above. Our algorithm, which we call a single-loop proximal-conditional-gradient penalty method ($\Alg$), is presented as Algorithm~\ref{alg1} below, where \eqref{get_xk} corresponds to the proximal-gradient step, \eqref{get_uk} and \eqref{get_yk} correspond to applying one step of the conditional gradient algorithm to minimizing $\widetilde F_{\beta_t}(x^{t+1},\cdot)$, and the update of $H_t$ in \eqref{definition_Ht} is designed to cater for the H\"older continuity of $\nabla f_1$; in particular, when $\nabla f_1$ is Lipschitz continuous so that one chooses $\mu = 1$, it holds that $H_t \equiv \max\{H_0, M_f\}$ for all $t\ge 1$. \add{The parameter $\delta$ controls how fast the penalty parameter grows, and will be chosen judiciously later to balance the objective
value deviations and the feasibility violations; see Remark~\ref{choose_delta}.}

\begin{algorithm}
	\caption{$\Alg$ for \eqref{problem} under Assumptions~\ref{Assumption 1} and~\ref{Assumption 2}.}
	\label{alg1}
	\begin{algorithmic}
		\STATE \textbf{Step 0.} Choose $x^0 \in \dom f$, $y^0 \in \dom g$, $\beta_0>0$, $H_0>0$, $\delta\in (0,1)$.
        Let $\lambda_{A}=\lambda_{\max}(A^*A)$.
		\STATE \textbf{Step 1.} For $t=0,1,\cdots$, let $\alpha_t = \frac2{t+2}$ and compute
            \begin{align}
                & R^t \!=\! Ax^t+By^t-c,\label{Rt}\\
                &x^{t+1} \!=\! \argmin_{x\in {\cal E}_1} \ \langle \nabla f_1(x^t) + \beta_{t}A^*R^t, x-x^t\rangle +\frac{H_t+\lambda_A \beta_t}{2}\|x-x^t\rVert^2 +f_2(x), \label{get_xk} \\
                & \widetilde R^t \!=\! Ax^{t+1} +By^t-c, \label{tRt}\\
                &u^t \!\in\! \Argmin_{y \in {\cal E}_2} \ \langle\nabla g_1(y^t) + \beta_{t}B^*\widetilde R^t,y\rangle +g_2(y), \label{get_uk}  \\
                &y^{t+1} \!=\! y^t +\alpha_t(u^t-y^t), \label{get_yk} \\
                &H_{t+1} \!=\! \max\left\{H_0, \frac{2M_f}{\mu+1} \right\}(t+1)^{1-\mu},\ \ \ \beta_{t+1} = \beta_0(t+2)^\delta. \label{definition_Ht}
            \end{align}
	\end{algorithmic}
\end{algorithm}

The convergence analysis of $\Alg$ will be presented in section~\ref{sec4}. For the rest of this section, we present some auxiliary lemmas.
\begin{lemma}   \label{lemma_of_h}
Consider \eqref{problem} and let $(x^*,y^*)$ be a solution to \eqref{problem}. Define $h(x,y)=\frac{1}{2}\|Ax+By-c\rVert^2$.
    Let $\delta \in (0,1)$, $\beta_0 > 0$, $\beta_{t+1}= \beta_0(t+2)^{\delta}$ and $\alpha_t=2/(t+2)$ for all $t \geq 0$. Then for any $x_1$, $x_2\in \mathcal{E}_1$, $y\in\mathcal{E}_2$ and $ t\ge 1 $, it holds that
    \begin{align}
        &(1-\alpha_t)(\beta_{t}-\beta_{t-1})h(x_1,y) + \alpha_t\beta_{t}h(x_1,y) + \alpha_t\beta_{t}\langle Ax_1+By-c,Ax^*-Ax_1\rangle \nonumber \\
        &+ \alpha_t\beta_{t}\langle Ax_2 +By -c, By^* -By\rangle \leq \alpha_t\beta_t \langle Ax_1 -Ax_2, By-By^* \rangle. \nonumber
    \end{align}
\end{lemma}
\begin{proof}
    First, notice that for every $t \geq 1$,\footnote{Notice that $\beta_t = \beta_0 (t+1)^\delta$ for all $t\ge 0$.}
    \begin{align}\label{relationship}
        &(1-\alpha_t)(\beta_{t}-\beta_{t-1}) -\alpha_t \beta_{t} = \frac{t\beta_0}{t+2}\left((t+1)^{\delta}
        -t^{\delta}\right) -\frac{2\beta_0}{t+2}(t+1)^{\delta} \nonumber \\
        & \overset{\mathop{(a)}}{\leq} \frac{t\beta_0}{t+2}\delta t^{\delta -1} -\frac{2\beta_0}{t+2}(t+1)^{\delta}
        \overset{\mathop{(b)}}{\leq} \frac{t^{\delta}\beta_0}{t+2} -\frac{2(t+1)^{\delta} \beta_0}{t+2} \leq -\frac{(t+1)^{\delta}\beta_0}{t+2} \leq 0,
    \end{align}
    where $(a)$ holds because $(\cdot)^{\delta}$ is concave on $\R_+$ and $(b)$ holds
    because $\delta \leq 1$.
    Therefore, we have $(1-\alpha_t)(\beta_{t}-\beta_{t-1})\leq \alpha_t \beta_{t}$.
    Hence,
    \begin{align}
        &(1-\alpha_t)(\beta_{t}-\beta_{t-1})h(x_1,y) + \alpha_t\beta_{t}h(x_1,y) + \alpha_t\beta_{t}\langle Ax_1+By-c,Ax^*-Ax_1\rangle \nonumber \\
        &+ \alpha_t\beta_{t}\langle Ax_2 +By -c, By^* -By\rangle  \nonumber \\
        &\overset{\mathop{(a)}}{\leq} 2\alpha_t\beta_{t} h(x_1,y) + \alpha_t\beta_{t}\langle Ax_1+By-c,Ax^*\!-\!Ax_1\rangle+ \alpha_t\beta_{t}\langle Ax_2 +By -c, By^* -By\rangle \nonumber \\
        &\overset{\mathop{(b)}}{=} \alpha_t\beta_t\|Ax_1+By-(Ax^*+ By^*)\rVert^2 +\alpha_t\beta_{t} \langle Ax_1 +By-(Ax^*+By^*),A(x^*-x_1) \rangle \nonumber\\
        &\ \ \ + \alpha_t\beta_{t}\langle Ax_2 +By-(Ax^*+By^*), B(y^*-y) \rangle \nonumber \\
        &= {\alpha_t\beta_{t}}\|A(x_1-x^*)\rVert^2 + {\alpha_t\beta_{t}}\|B(y-y^*)\rVert^2 +2\alpha_t\beta_{t}\langle Ax_1-Ax^*, By-By^*\rangle \nonumber \\
        & \ \ \ -\alpha_t\beta_{t}\|A(x_1-x^*)\rVert^2-\alpha_t\beta_{t}\|B(y-y^*)\rVert^2 +\alpha_t\beta_{t}\langle B(y-y^*), A(x^*-x_1) \rangle\nonumber\\
        & \ \ \  +\alpha_t\beta_{t}\langle A(x_2-x^*), B(y^*-y) \rangle  \nonumber \\
        &=\alpha_t\beta_t \langle Ax_1 -Ax_2, By-By^* \rangle, \nonumber
    \end{align}
    where $(a)$ holds because of \eqref{relationship}, $(b)$ holds because $Ax^*+ By^* =c$.
\end{proof}
\begin{lemma} \label{lemma_of_sum_Axk-Axk+1}
    Consider \eqref{problem} and suppose that Assumptions~\ref{Assumption 1} and~\ref{Assumption 2} hold. Let $(x^*,y^*)$ be a solution to \eqref{problem}.
    Suppose that  $\{(x^t,y^t)\}$ is generated by $\Alg$.  Then we have for each $k \ge 2$ that
    \begin{align}
        \left| \sum\limits_{t = 1}^{k-1}\alpha_t\beta_{t}(t\!+\!1)(t\!+\!2)\!\left\langle Ax^t\!-\! Ax^{t+1}\!,By^t\!-\!By^* \right\rangle \right|
        \!\leq\! 2^{\delta +3}\beta_0D_2 \!+\!\frac{16+8\delta}{1+\delta}\beta_0 D_2(k+1)^{1+\delta},  \nonumber
    \end{align}
    where $D_2 =\sup\limits_{x \in \dom f, y\in \dom g} \left| \langle Ax,By \rangle \right| < \infty$.
\end{lemma}
\begin{proof}
    Recall that $\dom f$ and $\dom g$ are bounded by Assumption~\ref{Assumption 1}. Therefore, we have $D_2 < \infty$.
    Next, write $y_e^t = y^t -y^*$ for notational simplicity. Then letting $a^t = Ax^t$ and $b^t = \alpha_t\beta_t(t+1)(t+2)By_e^t$ in \eqref{Abel_formula_vec}, we have for all $k\ge 2$ that
    \begin{align}
        &\left| \sum\limits_{t = 1}^{k-1}\alpha_t\beta_{t}(t+1)(t+2)\left\langle Ax^t- Ax^{t+1},By_e^t \right\rangle \right|   \nonumber \\
        &= \Big|6\alpha_1\beta_1 \left\langle Ax^1, By_e^1 \right\rangle - \alpha_{k-1}\beta_{k-1}k(k+1)\langle Ax^k, By_e^{k-1} \rangle \nonumber \\
       & \ \ \ + \sum_{t=1}^{k-2} \langle Ax^{t+1},  \alpha_{t+1}\beta_{t+1}(t+2)(t+3)By_e^{t+1} - \alpha_{t}\beta_{t}(t+1)(t+2)By_e^t \rangle \Big| \nonumber \\
        &\overset{(a)}{=}\!\Big| 2^{\delta +2}\beta_0\! \left\langle Ax^1, By_e^1 \right\rangle \!-\! 2\beta_0k^{1+\delta}\!\langle Ax^k, By_e^{k-1} \rangle  \!+\! \sum_{t=1}^{k-2}\!2\beta_0(t+1)^{1+\delta}\!\langle Ax^{t+1}, By^{t+1} \!-\! By^t   \rangle  \nonumber \\
       & \ \ \  + \sum_{t=1}^{k-2} \left(2\beta_0(t+2)^{1+\delta}- 2\beta_0(t+1)^{1+\delta}\right)\langle Ax^{t+1},  By_e^{t+1} \rangle \Big|\nonumber \\
       & \leq 2^{\delta +2}\beta_0 \left|\left\langle Ax^1, By_e^1 \right\rangle \right| + 2\beta_0k^{1+\delta}\left|\langle Ax^k, By_e^{k-1} \rangle\right|\nonumber\\
       &\ \ \ + \sum_{t=1}^{k-2}2\beta_0(t+1)^{1+\delta}\!\left|\langle Ax^{t+1}, By^{t+1} - By^t   \rangle\right| \nonumber \\
       & \ \ \  + \sum_{t=1}^{k-2} \left(2\beta_0(t+2)^{1+\delta}- 2\beta_0(t+1)^{1+\delta}\right)\left|\langle Ax^{t+1},  By_e^{t+1} \rangle \right| \nonumber \\
       &\overset{(b)}{\leq}  2^{\delta +3}\beta_0 D_2 + 4\beta_0k^{1+\delta}D_2  + \sum_{t=1}^{k-2}2\beta_0(t+1)^{1+\delta}\left|\langle Ax^{t+1}, B(y^{t+1}-y^t)   \rangle\right| \nonumber \\
       & \ \ \  + \sum_{t=1}^{k-2} \left(4\beta_0(t+2)^{1+\delta}- 4\beta_0(t+1)^{1+\delta}\right)D_2 \nonumber \\
        &\overset{(c)}{\le}  2^{\delta +3}\beta_0 D_2 + 4\beta_0k^{1+\delta}D_2  + \sum_{t=1}^{k-2}2\beta_0\alpha_t(t+1)^{1+\delta}\left|\langle Ax^{t+1}, B(u^t-y^t)   \rangle\right| \nonumber \\
       & \ \ \  +4\beta_0(1+\delta) \sum_{t=1}^{k-2}(t+2)^{\delta} D_2 \nonumber \\
       &\overset{(d)}{\leq}  2^{\delta +3}\beta_0 D_2 + 4\beta_0k^{1+\delta}D_2 +8\beta_0 \sum_{t=1}^{k-2}(t+2)^{\delta}D_2   +4\beta_0(1+\delta) \sum_{t=1}^{k-2}(t+2)^{\delta} D_2  \nonumber \\
        &\overset{(e)}{\leq}   2^{\delta +3}\beta_0D_2 + 4\beta_0 k^{1+\delta}D_2  + 8\beta_0D_2 \int_0^{k+1} x^{\delta} dx+4\beta_0(1+\delta)D_2\int_0^{k+1} x^{\delta} dx
        \nonumber \\
        &=  2^{\delta +3}\beta_0D_2+ \frac{8\beta_0}{1+\delta}D_2 (k+ 1)^{1+\delta}  + 4\beta_0 k^{1+\delta}D_2 +4\beta_0(k+1)^{1+\delta}D_2
        \nonumber \\
        &\leq   2^{\delta +3}\beta_0D_2 + 8\beta_0(k+1)^{1+\delta}D_2 + \frac{8\beta_0}{1+\delta}D_2 (k+1)^{1+\delta} \nonumber \\
         &= 2^{\delta+ 3}\beta_0D_2 +\frac{(16+8\delta)\beta_0}{1+\delta}D_2(k+1)^{1+\delta}, \nonumber
    \end{align}
    where $(a)$ holds because $\alpha_t = 2/(t+2)$ and $\beta_t = \beta_0(t+1)^{\delta}$, $(b)$ holds due to the definition of $D_2$ and hence $\left|\langle Ax^1, By_e^1 \rangle\right| = |\langle Ax^1, B(y^1-y^*) \rangle| \leq 2D_2$, $\left|\langle Ax^k, By_e^{k-1} \rangle\right| \leq 2D_2$ and $\left|\langle Ax^{t+1},  By_e^{t+1} \rangle \right| \leq 2D_2$, $(c)$ follows from \eqref{get_yk} and the convexity of $(\cdot)^{1+\delta}$ on $\R_+$, $(d)$ holds because $\alpha_t (t+1)^{1+\delta} \leq \frac{2}{t+2}(t+2)^{1+\delta} =2(t+2)^{\delta}$ and $\left|\langle Ax^{t+1}, B(u^t-y^t)   \rangle\right|  \leq 2D_2$, $(e)$ holds because $(\cdot)^{\delta}$ is increasing on $\R_+$.
\end{proof}

\section{Convergence analysis}\label{sec4}
In this section, we will first establish the global convergence of $\Alg$ by explicitly deriving its iteration complexity, and then study the local convergence rate of $\Alg$ based on the KL property and exponents.

\subsection{Global convergence and iteration complexity} \label{section_convergence_rate}
The main theorem in this subsection concerns the iteration complexity of $\Alg$ in terms of objective value deviations and feasibility violations, which is presented as Theorem~\ref{rate} below.

\begin{theorem}[Global convergence and complexity] \label{rate}
      Consider \eqref{problem}. Suppose that  Assumptions~\ref{Assumption 1}, \ref{Assumption 2} and the {\bf CQ} in Definition~\ref{CQ} hold. Let $(x^*, y^*)$ be a solution to \eqref{problem} and $\bar{\lambda}$ be defined in Lemma~\ref{CQ}.
      Suppose that $\{(x^t,y^t)\}$ is generated by $\Alg$. Then for all $t \ge 2$,
      $$\left|f(x^t)+ g(y^t) -f(x^*)-g(y^*) \right|
        \leq  \max\left\{ \tau_t, \|\bar\lambda\|\cdot{\cal G}_t \right\},$$
        $$
        \|Ax^t+By^t-c\rVert \leq {\cal G}_t,
        $$
        where
        \begin{equation}\label{taut}
          \tau_t= \begin{cases}
                    \displaystyle\frac{\omega_1}{t(t+1)}+\frac{\omega_2}{(t+1)^{1-\delta}}+\frac{\omega_3}{(t+1)^{\nu}}+ \frac{\omega_4}{(t+1)^{\mu}} &\ {\rm if}\ \mu \in (0,1), \vspace{0.1 cm}\\
                    \displaystyle\frac{\omega_1}{t(t+1)}+\frac{\omega_2}{(t+1)^{1-\delta}}+\frac{\omega_3}{(t+1)^{\nu}} +\frac{\omega_5}{t+1} &\ {\rm if}\ \mu =1,
                  \end{cases}
        \end{equation}
        \begin{equation}\label{Gt}
        {\cal G}_t = \frac{\|\bar{\lambda}\rVert}{\beta_0t^{\delta}} + \sqrt{\frac{\|\bar{\lambda}\rVert^2}{\beta_0^2t^{2\delta}}+\frac{2\tau_t}{\beta_0t^{\delta}}},
        \end{equation}
        \begin{align}
      &\omega_1 = 2^{\delta +3}\beta_0D_2 + \vartheta,\ \ \omega_2= 2\lambda_{A} D_f^2\beta_0+ 2\lambda_{B} D_g^2\beta_0+\frac{32+16\delta}{1+\delta}D_2\beta_0,\label{w1w2w3}\\
      &\omega_3 = \frac{2^{\nu+1}}{\nu+1}M_gD_g^{\nu+1},\ \ \omega_4 = 2\widetilde{H}_0D_f^2 + 2\omega_0,\ \  \omega_5 = 2\widetilde{H}_0D_f^2,\label{w3w4}\\
      &{with}\ \
    \omega_0 = 4\widetilde{H}_0\left(\frac{2M_f}{(1+\mu)\widetilde{H}_0}\right)^{\frac{2 }{1-\mu}}\ { and}\ \
    \widetilde{H}_0 = \max\left\{H_0, \frac{2M_f}{\mu+1} \right\},\label{w4w0H0}\\
    & \ \ \ \ \ \ \ \vartheta = 2(f(x^1) + g(y^1) + (\beta_0/2)\|Ax^1 + By^1-c\|^2 - f(x^*) - g(y^*)),
    \end{align}
    $M_f$, $M_g$, $\mu$ and $\nu$ are given in Assumption \ref{Assumption 2},
    $D_f$ and $D_g$ are defined in \eqref{Definition_of_diameter},
    $\lambda_{A} =\lambda_{\max}(A^*A)$, $\lambda_{B} = \lambda_{\max}(B^*B)$ and $D_2 = \sup\limits_{x \in \dom f, y\in \dom g}\left| \langle Ax,By\rangle \right| < \infty$.\footnote{Note that $D_2 < \infty$ thanks to the boundedness of $\dom f$ and $\dom g$.}
    \end{theorem}
\begin{proof}
    \add{Using Lemma~\ref{theorem_f+beta_h}, we have the following inequality for all $t\ge 2$,
    \begin{align}\label{hahaheheha}
        &f(x^{t}) + g(y^{t}) +\frac{\beta_{t-1}}{2}\|Ax^t+By^t-c\rVert^2 - f(x^*) -g(y^*) \le \tau_t,
    \end{align}
     where $\tau_t$ is defined in \eqref{taut}. In the remainder of the proof,
     we will discuss how the bounds on objective value deviations and feasibility violations along the sequence generated by $\Alg$ can be deduced from the above display.}

     \add{The argument is analogous to that of \cite[Theorem 2]{yu2020rc}. First, we invoke Lemma~\ref{optimal} to conclude that there exist $\xi_1 \in \partial f(x^*)$, $\xi_2\in \partial g(y^*)$ such that
    $0 = \xi_1 + A^*\bar{\lambda}$ and $0= \xi_2+B^*\bar{\lambda}$.
    Using this, we can obtain that
    \begin{align}
        0&= \langle \xi_1+A^*\bar{\lambda}, x^t -x^*\rangle+ \langle \xi_2+B^*\bar{\lambda}, y^t -y^*\rangle \nonumber \\
        & =\langle \xi_1, x^t -x^* \rangle +\langle \xi_2, y^t -y^*\rangle +\langle \bar{\lambda}, A(x^t -x^*) \rangle +\langle \bar{\lambda}, B(y^t -y^*) \rangle \nonumber \\
        &\overset{\mathop{(a)}}{\leq} f(x^t) -f(x^*) + g(y^t) -g(y^*) + \langle \bar{\lambda}, Ax^t +By^t -c\rangle, \nonumber
    \end{align}
    where $(a)$ holds because $f$ and $g$ are convex
    and $Ax^* + By^* = c$.}

    \add{From the above inequality, we deduce that
    \begin{align}
       -\|\bar{\lambda}\rVert \cdot\|Ax^t\! +\! By^t -c\rVert\leq -\langle \bar{\lambda}, Ax^t \!+\!By^t -c\rangle \leq f(x^t)\! -\!f(x^*) \!+\! g(y^t) \!-\!g(y^*). \label{F-F^*}
    \end{align}
    Using the above display, we can deduce further from \eqref{hahaheheha} that
    \begin{align}
        0&\leq f(x^t) -f(x^*) + g(y^t) -g(y^*) + \|\bar{\lambda}\rVert \cdot \|Ax^t +By^t -c\rVert \nonumber \\
        & \leq -\frac{\beta_{t-1}}{2}\|Ax^t +By^t -c\|^2 + \tau_t +\|\bar{\lambda}\rVert\|Ax^t +By^t -c\|. \nonumber
    \end{align}
    Solving this inequality for $\|Ax^t +By^t -c\|$, we have that
    \begin{align}
        \|Ax^t +By^t -c\rVert \leq \frac{\|\bar{\lambda}\rVert+\sqrt{\|\bar{\lambda}\rVert^2+2\beta_{t-1} \tau_t}}{\beta_{t-1}}
        =\frac{\|\bar{\lambda}\rVert}{\beta_{t-1}}+\sqrt{\frac{\|\bar{\lambda}\rVert^2}{\beta_{t-1}^2}+\frac{2\tau_t}{\beta_{t-1}}}. \label{Axk-Byk-c}
    \end{align}}

    \add{Finally, based on \eqref{F-F^*} and \eqref{Axk-Byk-c}, we obtain that
    \begin{align}
        f(x^t) + g(y^t) -f(x^*)- g(y^*) \geq -\frac{\|\bar{\lambda}\rVert^2}{\beta_{t-1}}- \|\bar{\lambda}\rVert\sqrt{\frac{\|\bar{\lambda}\rVert^2}{\beta_{t-1}^2}+\frac{2\tau_t}{\beta_{t-1}}}.
        \label{f+g-f*-g*_geq}
    \end{align}
    On the other hand, notice that \eqref{hahaheheha} implies $f(x^t)+ g(y^t) -f(x^*)-g(y^*)\leq \tau_t$.
    Combining this inequality with \eqref{Axk-Byk-c} and \eqref{f+g-f*-g*_geq} and noting that $\beta_{t-1} = \beta_0t^{\delta}$, we obtain the desired result.}
\end{proof}



    \begin{remark}[Asymptotic bounds] \label{remark_rate}
    In both cases {\rm (i)} and {\rm (ii)} of Theorem~\ref{rate}, one can show that
        $\tau_t = \mathcal{O} \left( \max\{(t+1)^{-(1-\delta)}, (t+1)^{-\nu}, (t+1)^{-\mu} \}\right)$. Let $\varpi_1 = \min \left\{ 1-\delta, \nu, \mu  \right\}$ and $\varpi_2 = \min \left\{ \delta,  \frac{1}{2}, \frac{\nu+ \delta}{2}, \frac{\mu+\delta}{2}  \right\}$. Then, by Theorem~\ref{rate},  we obtain
        $$\left|f(x^t)+ g(y^t) -f(x^*)-g(y^*) \right| = \mathcal{O}\left( (t+1)^{-\min\{\varpi_1,\varpi_2\}}  \right)$$
        and
        $$\|Ax^t+ By^t- c\| =\mathcal{O}\left( (t+1)^{-\varpi_2}\right).$$

    \end{remark}

\begin{remark}[Choosing $\delta$] \label{choose_delta}
We discuss how the parameter $\delta \in (0,1)$  in $\Alg$ can be chosen based on Remark \ref{remark_rate}. Intuitively, we should choose a suitable $\delta$ to ``balance" the values of $\varpi_1$ and $\varpi_2$ defined in Remark \ref{remark_rate}. We now present our suggested choice of $\delta$ according to the range of values of $\min \{\mu,\nu\}$.

\textbf{Case 1: } If  $\min \{ \mu, \nu \} \geq 0.5$, we can choose $\delta = 0.5$. Then we have
$$\varpi_1 = \min \left\{ 1-\delta, \nu, \mu\right\} = \min \left\{ 0.5, \nu, \mu  \right\} =\min \left\{ 0.5, \min \{ \mu, \nu \}  \right\} = 0.5,$$
and
$$\varpi_2 = \min \left\{ \delta,  0.5,  (\min \{ \mu, \nu \}+\delta)/2  \right\} = 0.5.$$
In this case, we have
$$
\left|f(x^t)+ g(y^t) -f(x^*)-g(y^*) \right| \!=\! \mathcal{O} \left( (t+1)^{-1/2} \right), ~\|Ax^t+ By^t- c\| \!=\!\mathcal{O}\left( (t+1)^{-1/2} \right).
$$

\textbf{Case 2: } If  $\min \{ \mu, \nu \} < 0.5$, we can choose  $\delta = 1-\min \{ \mu, \nu \} $. Then,
$$\varpi_1  \!=\!\min \left\{ 1-\delta, \min \{ \mu, \nu \}  \right\} = \min \{ \mu, \nu \}, ~\varpi_2 \!=\! \min \left\{ \delta,  0.5,  (\min \{ \mu, \nu \}+\delta)/2  \right\} = 0.5.$$
In this case, we have
$$
\left|f(x^t)\!+\! g(y^t) \!-\!f(x^*)\!-\!g(y^*) \right| \!=\! \mathcal{O} \!\left(\! (t+1)^{-\min \{ \mu, \nu \}} \!\right)\!, ~\|Ax^t+ By^t- c\| \!=\!\mathcal{O}\!\left(\!(t+1)^{-1/2}\!\right)\!.
$$

\end{remark}

\subsection{KL property and local convergence rate to the solution set} \label{section_KL property}

In this subsection, we study the local convergence rate of the sequence generated by $\Alg$ to the solution set.
We first show that, under some structural assumptions on $f$ and $g$ in \eqref{problem} and an assumption on the KL property of the extended objective of \eqref{problem} (i.e., the sum of the objective and the indicator function of the constraint set), the distance to the set of minimizers can be related to objective value deviation (i.e., $f(x) + g(y) - \inf_{Ax + By = c}\{f(x) + g(y)\}$) and feasibility violation (i.e., $\|Ax+By-c\|$). This together with Theorem~\ref{rate} will allow us to derive an explicit asymptotic convergence rate as a corollary.

We now present our theorem concerning bounds on the distance to the set of minimizers.
\begin{theorem}[Bounding the distance to minimizers] \label{Theorem_of_KL}
    Let $h: {\cal E}_1 \rightarrow (-\infty,\infty]$ be a proper closed convex function, $G: {\cal E}_1 \rightarrow {\cal E}$ be a linear map and $b \in G \, {\rm ri} \dom h$.
    Suppose further that $h(x) = h_0(x) + \delta_{\Theta}(x)$, where $\Theta$ is a compact convex set and $h_0$ is a real-valued convex function.
    Let $H(x) = h(x) +\delta_{\{b\}}(Gx) $. If $H$ is a KL function  with exponent $\alpha\in [0,1)$, then there exist $\epsilon>0$, $c_0>0$ and $\eta>0$ such that
    \begin{align}\label{Fineq}
        \dist\left(x, \Argmin H\right) \leq c_0\left|h(x)+\eta\|Gx-b\|-\inf H\right|^{1-\alpha}
    \end{align}
    whenever $\dist(x,\Argmin H) \le \epsilon$.
\end{theorem}

    \begin{proof}
        We start by establishing four auxiliary facts. First, since $b \in G \, {\rm ri} \dom h = G\,{\rm ri}\,\Theta$, using  \cite[Corollary 3]{bauschke1999strong} and the compactness of $\Theta$, we conclude that there exists a $\kappa>0$ such that
        \begin{align}
            \mathrm{dist}(x,\Theta \cap G^{-1}\{b\}) \leq \kappa \mathrm{dist}(x, G^{-1}\{b\})\ \ \ \forall x\in \Theta.
            \label{corollary_bauschke}
        \end{align}

        Second, notice that $H(x) = h(x)+ \delta_{\{b\}}(Gx) = h_0(x) +\delta_{\Theta \cap G^{-1}\{b\}}(x)$ is level-bounded thanks to the compactness of $\Theta$. Consequently, $\Argmin H$ is nonempty, closed and convex. Moreover, for any $x$, it holds that
        \begin{align}\label{PCDineq}
        &{\rm dist}(P_{\Theta\cap G^{-1}\{b\}}(x),\Argmin H)\le \|P_{\Theta\cap G^{-1}\{b\}}(x) - P_{\Argmin H}(x)\|\nonumber\\
        & \overset{(a)}= \|P_{\Theta\cap G^{-1}\{b\}}(x) - P_{\Theta\cap G^{-1}\{b\}}P_{\Argmin H}(x)\|\nonumber\\
        &\overset{(b)}\le \|x - P_{\Argmin H}(x)\|= {\rm dist}(x,\Argmin H),
        \end{align}
     where $(a)$ holds because $\Argmin H\subseteq \Theta\cap G^{-1}\{b\}$ and $(b)$ holds because projections onto closed convex sets are nonexpansive.

        Third, notice that $h_0$ is convex and real-valued, and hence it is locally Lipschitz continuous.
        Since $\Theta$ is compact, there exists $L_0 > 0$ such that
     \begin{equation}\label{Lip}
    |h_0(x) - h_0(u) | \leq L_0\|x - u\| \ \ {\rm whenever }\ \ x, u\in \Theta.
     \end{equation}

        Last but not least, since $H$ is a level-bounded KL function with exponent $\alpha$, we have from \cite[Lemma~3.10]{YuConvergenceRate2021} the existence of $\bar{c} >0$, $\epsilon > 0$, $r_0>0$ such that
        \begin{align}
            \mathrm{dist} (x,\Argmin H) \leq \bar{c}(H(x) - \inf H)^{1-\alpha} \label{property_of_error_bound}
        \end{align}
        whenever $x\in {\Theta} \cap G^{-1}\{b\}\, (=\dom \partial H)$ satisfies $\mathrm{dist}(x,\Argmin H) \leq \epsilon$ and
        $\inf H\le H(x)<\inf H + r_0$. Since $H$ is continuous on its (compact) domain, by shrinking $\epsilon$ further if necessary, we will assume that \eqref{property_of_error_bound} holds when $x\in {\Theta} \cap G^{-1}\{b\}$ and $\mathrm{dist}(x,\Argmin H) \leq \epsilon$.

     We are now ready to establish \eqref{Fineq}. Let $\epsilon_1 = \min\{\epsilon,1\}$. Pick any $x$ that satisfies $h(x) < \infty$\footnote{Notice that \eqref{Fineq} holds trivially if $h(x) =\infty$.} and $\mathrm{dist}(x,\Argmin H)\leq \epsilon_1$. Then $x\in \Theta$ and we have
        \begin{align}
            &\mathrm{dist}(x, \Argmin H) \leq \mathrm{dist}(P_{\Theta\cap G^{-1}\{b\}}(x), \Argmin H)
            + \mathrm{dist}(x, \Theta\cap G^{-1}\{b\}) \nonumber \\
            & \overset{\mathop{(a)}}{\leq} \bar{c}(H(P_{\Theta\cap G^{-1}\{b\}}(x)) - \inf H)^{1-\alpha} + \mathrm{dist}(x,\Theta\cap G^{-1}\{b\}) \nonumber \\
            & \overset{(b)}{\leq}\bar{c}(h_0(P_{\Theta\cap G^{-1}\{b\}}(x))-\inf H)^{1-\alpha} + \mathrm{dist}(x,\Theta\cap G^{-1}\{b\})^{1-\alpha}  \nonumber \\
            & \overset{\mathop{(c)}}{\leq} \bar{c}\big(h_0(x)+L_0\mathrm{dist}(x, \Theta\cap G^{-1}\{b\})-\inf H\big)^{1-\alpha} + \mathrm{dist}(x,\Theta\cap G^{-1}\{b\})^{1-\alpha} \nonumber \\
            & \overset{(d)}\le  \bar{c}\left(\bigg(h_0(x)+L_0\kappa\mathrm{dist}(x,  G^{-1}\{b\})-\inf H\bigg)^{1-\alpha}
            + \frac{\kappa^{1-\alpha}}{\bar{c}} \mathrm{dist}(x,  G^{-1}\{b\})^{1-\alpha} \right)\nonumber \\
            &\overset{(e)}{\leq} 2^{\alpha}\bar{c}\left(h_0(x)+\left(L_0\kappa +  \frac{\kappa}{\bar{c}^{1/(1-\alpha)}} \right)\mathrm{dist}(x, G^{-1}\{b\})-\inf H   \right)^{1-\alpha} \nonumber \\
            &\leq  2^{\alpha}\bar{c}\left(h_0(x)+\bar{c}_1\left(L_0\kappa +  \frac{\kappa}{\bar{c}^{1/(1-\alpha)}} \right)\|Gx  -b\|-\inf H   \right)^{1-\alpha},
            \nonumber
        \end{align}
        where $(a)$ holds because of \eqref{property_of_error_bound} and \eqref{PCDineq} (note that one can deduce from \eqref{PCDineq} that $\mathrm{dist}(P_{\Theta\cap G^{-1}\{b\}}(x),\Argmin H)\leq \epsilon_1)$,
        $(b)$ holds because we have $\mathrm{dist}(x,\Theta\cap G^{-1}\{b\}) \le \dist(x,\Argmin H) \le \epsilon_1 \leq 1$,
        $(c)$ follows from \eqref{Lip}, $(d)$ follows from \eqref{corollary_bauschke} and the fact that $x\in \Theta$, and $(e)$ holds because
        $a^{1-\alpha}+ b^{1-\alpha} \leq 2^{\alpha} (a+b)^{1-\alpha}$ for all $a\ge 0$, $b\ge 0$. Finally,
        the last inequality holds for some constant $\bar c_1>0$ (independent of $x$) thanks to \cite[Lemma 3.2.3]{facchinei2003finite}.
    \end{proof}

    The next corollary concerning local convergence rate of the sequence generated by $\Alg$ to the solution set of \eqref{problem} is now an immediate consequence of Theorem~\ref{rate} and Theorem~\ref{Theorem_of_KL}.
    \begin{corollary}[Local convergence rate to $\Argmin F$] \label{local_rate}
        Consider \eqref{problem}. Suppose that Assumption \ref{Assumption 2} and the {\bf CQ} in Definition~\ref{CQ}
        holds. Suppose further that $f(x) = f_0(x) + \delta_{\Xi}(x)$ and $g(y) = g_0(y) + \delta_{\Delta}(y)$, where
    $\Xi$, $\Delta$ are compact convex sets and $f_0$ and $g_0$ are real-valued convex functions.
    Let $F(x,y)= f(x)+ g(y)+ \delta_{\{c\}}(Ax+By)$ and
    let $\{(x^t,y^t)\}$ be generated by $\Alg$.
        If $F$ is a KL function with exponent $\alpha\in [0,1)$, then
        \begin{align}
            \mathrm{dist}\left((x^t, y^t), \Argmin F\right) = \mathcal{O} \left((t+1)^{-(1-\alpha)\min\{\varpi_1, \varpi_2\}} \right), \nonumber
        \end{align}
        where $\varpi_1$ and $\varpi_2$ are defined in Remark~\ref{remark_rate}.
    \end{corollary}
    \begin{proof}
    First, from Theorem \ref{Theorem_of_KL}, we know that there exist $\epsilon > 0$, $c_0>0$ and $\eta>0$ such that
        \begin{align}\label{haha}
            \mathrm{dist}\left((x ,y), \Argmin F\right) \leq  c_0 \left|f(x) + g(y)+\eta\|Ax +By -c\| -\inf F \right|^{1-\alpha}
        \end{align}
    whenever ${\rm dist}((x,y),\Argmin F)\le \epsilon$.

    Now, notice that $\{(x^t, y^t)\}$ is bounded as it is contained in the compact set $\Xi\times \Delta$.
    Let $\cal S$ be the set of cluster points of $\{(x^t, y^t)\}$. Then there exists $T>0$ such that
        \begin{align*}
            \dist((x^t, y^t), {\cal S}) \leq \epsilon ~~~\forall t >T.
        \end{align*}
        Since we have ${\cal S} \subseteq \Argmin F$ in view of Remark \ref{remark_rate} and the continuity of $f_0$ and $g_0$, we deduce further that for all $t>T$,
            \begin{align*}
            \dist((x^t, y^t), \Argmin F) \leq \dist((x^t, y^t), {\cal S}) \leq \epsilon.
        \end{align*}
        Using this and \eqref{haha}, we conclude that for all $t> T$,
        \begin{align}
            \mathrm{dist}\left((x^t ,y^t), \Argmin F\right) &\leq  c_0 \left|f(x^t) + g(y^t)+\eta\|Ax^t +By^t -c\|-\inf F \right|^{1-\alpha} \nonumber \\
            & \leq  c_0\left( \left|f(x^t) + g(y^t)-\inf F \right| +\eta\|Ax^t +By^t -c\| \right)^{1-\alpha}. \nonumber
        \end{align}
    The desired result follows upon combining the above display with Remark~\ref{remark_rate}.
    \end{proof}

    \subsubsection{\add{Deducing KL exponents}}
    In view of Corollary~\ref{local_rate}, deducing the KL exponent of the function $F(x,y) := f(x)+ g(y)+ \delta_{\{c\}}(Ax+By)$ \add{(i.e., the extended objective of \eqref{problem})} is the key to deriving the local convergence rate of $\{\dist((x^t,y^t),\Argmin F)\}$. \add{Here, we discuss two strategies for doing so. The first strategy is based on a systematic framework described in \cite[Section~5]{lindstrom2024optimal} that leverages (i) conic reformulation; (ii) facial reduction techniques for deducing error bounds; (iii) the inf-projection calculus rule for KL exponents; and (iv) the interplay between KL exponents and error bounds. We illustrate this strategy in the following examples.}

    \begin{example}\label{Example4.1}
   \add{Consider \eqref{exam_reform}. Recall that by construction, the solution set of \eqref{exam_prob1} is contained in the interior of the set $\{x\in \R^n: \|x\|_\infty \leq 1+\|\widehat x\|_1 \}$. Thus, the following function has the same KL exponent as the extended objective of \eqref{exam_reform}:}
    \begin{equation}\label{Fdef2}
    F(x, y) := \|x\|_1+\delta_{\|\cdot\|_p \leq \sigma}(y) +\delta_{\{b\}}(Ax-y).
    \end{equation}

    \add{We now deduce the KL exponent of \eqref{Fdef2}.} To this end, notice that we can rewrite \eqref{exam_reform} as follows.
    \begin{equation} \label{KL_example_rewite}
      \begin{array}{cl}
             \min\limits_{x,w,y,s} & w\\
            {\rm{s.t.}} &  s = \sigma,\ \  Ax-y = b,  \\
                &   (y,s) \in \mathcal{K}_p^{{\color{blue}m+1}},\ \ (x,w) \in \mathcal{K}^{n+1}_1,
        \end{array}
    \end{equation}
    where
    \[
    \mathcal{K}_p^{\color{blue}m+1} =\{(y,s)\in \R^{\color{blue}m}\times \R_+: \|y\|_p \leq s \}  \ {\rm and}\  \mathcal{K}_1^{n+1} =\{(x,w)\in \R^n\times \R_+: \|x\|_1 \leq w \}
    \]
    are the $p$-cone and $L_1$-norm cone, respectively.

    Define $z = (x,w,y,s)$ for notational simplicity. Then the feasible set of  \eqref{KL_example_rewite} is
    \begin{equation}\label{frakF}
   \add{ {\cal F}_{C}} := \{z = (x,w,y,s)\in \mathcal{K}^{n+1}_1 \times \mathcal{K}_p^{{\color{blue}m+1}}: s= \sigma,\ Ax-y =b \}.
    \end{equation}
    Let the optimal value of  \eqref{exam_prob1} be $\theta$. Then the solution set of  \eqref{KL_example_rewite} is
    \[
    {\cal S} :=\underbrace{\{z :\; w = \theta, ~ s=\sigma, ~Ax -y =b \}}_{{\cal S}_1}
\cap \underbrace{(\mathcal{K}^{n+1}_1\times \mathcal{K}_p^{{\color{blue}m+1}})}_{{\cal S}_2}.
    \]
    We have the following observations concerning ${\cal S}$.
    \begin{itemize}
      \item Notice that $\mathcal{K}^{n+1}_1$ is polyhedral and all proper exposed faces of $\mathcal{K}^{{\color{blue}m+1}}_p$ are polyhedral but $\mathcal{K}^{{\color{blue}m+1}}_p$ is not polyhedral; see \cite[Section~4.1]{lindstrom2024optimal}.
    Then we have $\ell_{\rm poly}(\mathcal{K}^{n+1}_1) = 0$ and $\ell_{\rm poly}(\mathcal{K}^{{\color{blue}m+1}}_p) = 1$; see \cite[Section~5.1]{lourencco2018facial} for the definition of $\ell_{\rm poly}$. In view of this, when we apply \cite[Proposition~3.2]{lindstrom2023error} to the feasibility problem of finding an element in ${\cal S}_1\cap {\cal S}_2$, we see that the $\ell$ there is at most $2$,
    \add{where $\ell$ is the length of the chain of faces in \cite[Eq.~(3.1)]{lindstrom2023error}.}
      \item Using \cite[Proposition~3.13]{lindstrom2023error}, the discussion in \cite[Section~4.2]{lindstrom2024optimal} and the Hoffman error bound, one can deduce that the so-called one-step facial residual functions ($\mathds{1}$-FRFs) of $(\mathcal{K}^{n+1}_1\times \mathcal{K}_p^{{\color{blue}m+1}})$ takes the form of $\psi(\epsilon,\eta) = \rho(\eta)\epsilon + \hat\rho(\eta)\epsilon^\frac12$ for some nondecreasing functions $\rho$ and $\hat \rho$, where $\epsilon \ge 0$ and $\eta\ge 0$.
    \end{itemize}
    The above observations  together with \cite[Lemma~2.1]{lindstrom2024optimal} show that $\{{\cal S}_1,{\cal S}_2\}$ satisfies a uniform H\"olderian error bound with exponent $\frac12$. In particular, this means that \add{ for every bounded set ${\cal B}$, there exists $c_{\cal B} > 0$ such that}
    \[
    {\rm{dist}}(z, {\cal S}_1\cap {\cal S}_2) \le \add{c_{\cal B}}\max\{{\rm{dist}}(z, {\cal S}_1)^\frac12,{\rm{dist}}(z, {\cal S}_2)^\frac12\} \ \ \ \forall z \in \add{{\cal B}}.
    \]
    Thus, there exists \add{$\kappa_{\cal B} > 0$} such that
    \begin{align}
      {\rm{dist}}(z, {\cal S}_1\cap {\cal S}_2) & \leq \add{c_{\cal B}} {\rm{dist}}(z, {\cal S}_1)^\frac12 \leq \add{\kappa_{\cal B}} |w - \theta|^{\frac{1}{2}}\ \ \ \forall z\in \add{{\cal B} \cap {\cal F}_{C}}. \nonumber
    \end{align}
    Combining this result with \cite[Theorem~5]{bolte2017error}, we conclude that  the function $\widehat F(z) := w+ \delta_{ \add{ {\cal F}_{C}}}(z)$ is a KL function with exponent $\frac{1}{2}$.

    Next, if we fix any $(x,y)$  \add{satisfying $Ax -y =b$ and $\|y\|_p \le \sigma$},
    and let $Y\!(x,y) \!=\! \Argmin_{w,s}\! \widehat F(x,w,y,s)$, then $Y(x,y) = \{\|x\|_1,\sigma\}$. Therefore, $Y(x,y)$ is nonempty and compact. Observe that $F(x, y)  = \inf_{w,  s} ~w + \delta_{ \add{ {\cal F}_{C}}}(z)= \inf_{w,s} \widehat F(z)$,
    where $F$ is given in \eqref{Fdef2} and $ \add{ {\cal F}_{C}}$ is defined in \eqref{frakF}. Then, in view of \cite[Corollary 3.3]{yu2022kurdyka} and the KL exponent of $\widehat F$, we conclude that $F$ is a KL function with exponent $\frac12$.
    \end{example}

\begin{example} \label{KL_hankel}
    \add{Consider \eqref{reform_prob3}.  By construction, the solution set of \eqref{reform_prob3} is contained in the interior of the set
    $\left\{\bm x\in \mathbb{C}^n: \|{\bm x} - \Pi_\Omega(\bar{\bm x})\|_2 \le 1+ \|\Pi_\Omega({\bar{\bm x}})\|_2 +\sigma \right\}$. Hence the following function has the same KL exponent as the extended objective of \eqref{reform_prob3}:
    \begin{align}\label{Fexample42}
      F(x_{_{\cal R}}, x_{_{\cal I}},Y_{_{\cal R}},Y_{_{\cal I}} )
      &= \sum_{j\in \Omega}w_j\sqrt{(x_{_{\cal R}} -\bar{x}_{_{\cal R}})_j^2 + (x_{_{\cal I}} - \bar{x}_{_{\cal I}})_j^2}  \notag\\
      &\ \ \ \ + \delta_{\|\cdot + i \cdot \|_* \le \sigma}(Y_{_{\cal R}},Y_{_{\cal I}}) + \delta_{\{0\}}(Y_{_{\cal R}} - {\cal H}(x_{_{\cal R}}))+ \delta_{\{0\}}(Y_{_{\cal I}} - {\cal H}(x_{_{\cal I}})).
    \end{align}}

     \add{We now deduce the KL exponent of the above function.
     We first note from \cite{recht2010guaranteed} that the nuclear norm of a complex matrix ${\bm Y} \in \mathbb{C}^{m\times q}$ can be represented as:\footnote{\add{We would like to point out that while the discussion in \cite{recht2010guaranteed} was for real matrices, its proof extends to complex matrices.}}
    \begin{equation*}
      \|{\bm Y}\|_* = \min_{{\bm W},{\bm V}} \left\{ \frac{1}{2} ({\rm tr}({\bm W})+ {\rm tr}({\bm V})): \left[ \begin{array}{cc}
                                                                                        {\bm W} & {\bm Y}^* \\
                                                                                        {\bm Y} & {\bm V}
                                                                                      \end{array} \right] \succeq 0, {\bm W} \in \mathbb{H}^q, {\bm V} \in \mathbb{H}^m \right\},
    \end{equation*}
    where $\mathbb{H}^m$ is the space of $m \times m$ Hermitian matrices and ${\bm Y}^*$ is the conjugate transpose of ${\bm Y}$.
    Therefore, problem \eqref{reform_prob3} can be equivalently reformulated as follows:
    \begin{equation}\label{KL_exam_reform2}
      \begin{array}{cl}
      \min\limits_{z\in \mathbb{E}} & t \\
      {\rm s.t.} &u =\Pi_{\Omega}\big(w\circ(x_{_{\cal R}} - \bar{x}_{_{\cal R}})\big),~v =\Pi_{\Omega}\big(w\circ(x_{_{\cal I}} - \bar{x}_{_{\cal I}})\big), \\
                      &Y_{_{\cal R}} = {\cal H}(x_{_{\cal R}}), ~ Y_{_{\cal I}}  = {\cal H}(x_{_{\cal I}}),  \\
                      &  \begin{bmatrix}
                      W_{_{\cal R}}+ iW_{_{\cal I}} & (Y_{_{\cal R}} + iY_{_{\cal I}})^* \\
                      Y_{_{\cal R}}+ iY_{_{\cal I}} & V_{_{\cal R}}+ iV_{_{\cal I}}
                      \end{bmatrix} \succeq 0,  \\ [8 pt]
                      & \frac{1}{2} ({\rm tr}(W_{_{\cal R}})+ {\rm tr}(V_{_{\cal R}}))+ \alpha = \sigma, ~\alpha \ge 0, \\
                      & t = \sum_{j=1}^n \tau_j,  \tau_j\ge  \sqrt{u_j^2 + v_j^2}, \ \text{for}\ j=1,...,n,
    \end{array}
    \end{equation}
    where
    \begin{align*}
    z &:= (u_1 , v_1,\tau_1,\ldots,u_n, v_n,\tau_n,Y_{_{\cal R}},Y_{_{\cal I}},W_{_{\cal R}},W_{_{\cal I}},V_{_{\cal R}},V_{_{\cal I}}, x_{_{\cal R}}, x_{_{\cal I}}, t,\alpha) \\
    &\in \underbrace{\mathbb{R}^3\times \cdots \times \mathbb{R}^3}_{n\ {\rm copies}} \times
            \mathbb{R}^{m\times q} \times \mathbb{R}^{m\times q} \times {\cal S}^q\!\times\! {\cal A}^q\!\times\! {\cal S}^m\!\times\! {\cal A}^m
            \times \mathbb{R}^n\times\R^n\times \R\times \mathbb{R} =: \mathbb{E},
    \end{align*}
    ${\cal S}^m$ is the space of $m\times m$ real symmetric matrices, ${\cal A}^m$ is the space of $m\times m$ real anti-symmetric matrices.
    Let $\theta$ denote the optimal value of \eqref{reform_prob3}.
    We define
    \[
        {\cal S}_1 \!=\! \Big\{ z\in \mathbb{E}:\!\!
        \begin{array}{l}
          u =\Pi_{\Omega}\big(w\circ(x_{_{\cal R}} -\bar{x}_{_{\cal R}})\big), v =\Pi_{\Omega}\big(w\circ(x_{_{\cal I}} - \bar{x}_{_{\cal I}})\big), t=\theta, t = \sum_{j=1}^n \tau_j, \\
                      Y_{_{\cal R}} = {\cal H}(x_{_{\cal R}}), ~ Y_{_{\cal I}} = {\cal H}(x_{_{\cal I}}),~\frac{1}{2} ({\rm tr}(W_{_{\cal R}})+ {\rm tr}(V_{_{\cal R}}))+ \alpha = \sigma
        \end{array}
           \!\!\Big\}
    \]
    and
    \[
        {\cal S}_2 = \underbrace{{\cal K}_2^3\times \cdots \times {\cal K}_2^3}_{n\ {\rm copies}} \times \mathbb{H}_+^{m+q} \times \mathbb{R}^n\times\R^n\times \R\times \mathbb{R}_+ \subset \mathbb{E},
    \]
    where ${\cal K}_2^3 = \{(a,b,c)\in \R^3: \sqrt{a^2+b^2}\le c\}$ and $\mathbb{H}_+^{m+q}\subset \mathbb{R}^{m\times q} \times \mathbb{R}^{m\times q} \times {\cal S}^q\!\times\! {\cal A}^q\!\times\! {\cal S}^m\!\times\! {\cal A}^m$ is defined as:
    \[
       \left\{ (Y_{_{\cal R}},Y_{_{\cal I}},W_{_{\cal R}},W_{_{\cal I}},V_{_{\cal R}},V_{_{\cal I}}): \begin{array}{l}
                      \begin{bmatrix}
                      W_{_{\cal R}}+ iW_{_{\cal I}} & (Y_{_{\cal R}} + iY_{_{\cal I}})^* \\
                      Y_{_{\cal R}}+ iY_{_{\cal I}} & V_{_{\cal R}}+ iV_{_{\cal I}}
                      \end{bmatrix} \succeq 0, \\[8pt]
                       W_{_{\cal R}} = W_{_{\cal R}}^T, V_{_{\cal R}} =V_{_{\cal R}}^T, W_{_{\cal I}} = -W_{_{\cal I}}^T, V_{_{\cal I}} = -V_{_{\cal I}}^T
                     \end{array}     \right\}.
    \]
    Then, the solution set of \eqref{KL_exam_reform2} is ${\cal S}= {\cal S}_1 \cap {\cal S}_2$.
    For ${\cal S}$, we have the following observations.}
    \begin{itemize}
      \item \add{
      As mentioned in Example~\ref{Example4.1}, $\ell_{\rm poly}({\cal K}_2^3) =1$. Since $\mathbb{H}_+^{m+q}$ is a symmetric cone, we have $\ell_{\rm poly}(\mathbb{H}_+^{m+q}) \le m+q - 1$ thanks to \cite[Remark~39]{lourencco2021amenable} and \cite[Theorem~28]{lourencco2021amenable}. Notice that $\ell_{\rm poly}(\mathbb{R}_+ ) = \ell_{\rm poly}(\mathbb{R}^n) = 0$. Therefore, by  \cite[Proposition~3.2]{lindstrom2023error}, we obtain that, $ \ell - 1 \le m+q+n-1$, where $\ell$ is the length of the chain of faces in \cite[Eq.~(3.1)]{lindstrom2023error}.
      }
      \item \add{By \cite[Proposition~3.13]{lindstrom2023error}, the $\mathds{1}$-FRFs derived in \cite[Theorem~35]{lourencco2021amenable} and the discussion in \cite[Section~4.2]{lindstrom2024optimal}, and the Hoffman error bound, we have that  $\mathds{1}$-FRFs of ${\cal S}_2$ can be taken as $\psi(\epsilon,\eta) = \rho(\eta)\epsilon + \hat\rho(\eta)\epsilon^\frac12$ for some nondecreasing functions $\rho$ and $\hat \rho$, where $\epsilon \ge 0$ and $\eta\ge 0$.
      }
    \end{itemize}
     \add{Combining the above observations with \cite[Lemma~2.1]{lindstrom2024optimal}, we deduce that $\{{\cal S}_1,{\cal S}_2\}$ satisfies a uniform H\"olderian error bound with exponent $2^{-\ell+1}$, i.e.,  for every bounded set ${\cal B}\subset \mathbb{E}$, there exists $c_{\cal B} > 0$ such that
    \[
    {\rm{dist}}(z, {\cal S}_1\cap {\cal S}_2) \le c_{\cal B}\max\{{\rm{dist}}(z, {\cal S}_1)^{2^{-\ell+1}},{\rm{dist}}(z, {\cal S}_2)^{2^{-\ell+1}}\}
     \ \ \forall z \in {\cal B}.
    \]
    Let ${\cal F}_H\subseteq \mathbb{E}$ be the feasible region of \eqref{KL_exam_reform2}. Then, there exists $\kappa_{\cal B} > 0$ such that
    \begin{align}
      {\rm{dist}}( z, {\cal S}_1\cap {\cal S}_2) & \leq c_{\cal B} {\rm{dist}}( z, {\cal S}_1)^{2^{-\ell+1}} \leq \kappa_{\cal B} |t - \theta|^{2^{-\ell+1}}\ \ \ \forall  z\in {\cal B} \cap {\cal F}_H. \nonumber
    \end{align}
    Combining this result with \cite[Theorem~5]{bolte2017error}, we see that $ \widetilde F( z) := t+ \delta_{{\cal F}_H}(z)$ is a KL function with exponent $1 - 2^{-\ell+1}$.}

    \add{Next,  we fix any $(x_{_{\cal R}}, x_{_{\cal I}}, Y_{_{\cal R}}, Y_{_{\cal I}})$ such that $Y_{_{\cal R}}+ iY_{_{\cal I}} = {\cal H}(x_{_{\cal R}}+ix_{_{\cal I}})$ and $\|Y_{_{\cal R}}+iY_{_{\cal I}}\|_* \le \sigma$, and let
    \[
        {\cal Y}(x_{_{\cal R}}, x_{_{\cal I}},Y_{_{\cal R}},Y_{_{\cal I}}) = \Argmin_{t, u, v, \tau, \alpha, W_{_{\cal R}}, W_{_{\cal I}}, V_{_{\cal R}}, V_{_{\cal I}}} \widetilde F( z).
    \]
    Then one can check that ${\cal Y}(x_{_{\cal R}}, x_{_{\cal I}},Y_{_{\cal R}},Y_{_{\cal I}})$ is nonempty and compact. Moreover, it holds that $ F(x_{_{\cal R}}, x_{_{\cal I}},Y_{_{\cal R}},Y_{_{\cal I}}) = \inf_{t, u, v, \tau, \alpha, W_{_{\cal R}}, W_{_{\cal I}}, V_{_{\cal R}}, V_{_{\cal I}}}  \widetilde F( z)$, where $F$ is defined in \eqref{Fexample42}.
     In view of \cite[Corollary 3.3]{yu2022kurdyka} and the KL exponent of $\widetilde F$, we see that the $ F$ in \eqref{Fexample42} is a KL function with exponent $1 - 2^{-\ell+1}$.}
\end{example}

In many applications, it may be difficult to calculate the KL exponent of the $F$ in Corollary~\ref{local_rate} via the facial reduction techniques described in the above example. As an alternative strategy, when {\bf CQ} in Definition~\ref{CQ} holds, we propose to derive such a KL exponent from the KL exponent of an associated function given by ${\frak L}(x,y):= f(x) + g(y) + \langle \bar{\lambda},Ax + By - c\rangle$, where $\bar \lambda$ is a Lagrange multiplier of \eqref{problem}.
Indeed, when $f$ and $g$ are continuous on their domains, the KL exponent of ${\frak L}$ can be deduced from those of $x \mapsto f(x) + \langle \bar{\lambda},Ax \rangle$ and $y\mapsto g(y) + \langle \bar{\lambda}, By\rangle $; see \cite[Theorem 3.3]{li2018calculus}.

We now present the following theorem concerning the KL exponent of the Lagrangian function.
\begin{theorem}[KL exponent from Lagrangian]
    Let $h: \mathcal{E}_1 \rightarrow (-\infty,\infty]$ be a proper closed convex function, $G: \mathcal{E}_1 \rightarrow \mathcal{E}$ be a linear map
    and $b \in G\, \dom h$. Let $H(x) = h(x) + \delta_{\{b\}}(Gx)$ and suppose that $\Argmin H\neq \emptyset$. Let $\bar{\lambda}$ be a Lagrange multiplier for the following problem\footnote{Recall that a Lagrange multiplier exists if we assume in addition that $b\in G\,{\rm ri}\,\dom h$.}
    \begin{align*}
    \begin{array}{cl}
      \min\limits_{x\in {\cal E}_1}&h(x)\\
      {\rm s.t.}& Gx = b.
    \end{array}
    \end{align*}
    Suppose that $H_{\bar{\lambda}}(x) := h(x) + \langle \bar{\lambda}, Gx - b\rangle$ satisfies the KL property with exponent $\alpha\in [0,1)$
    at an $\bar{x}\in \Argmin H$. If ${\rm ri}\, (\Argmin H_{\bar{\lambda}}) \cap G^{-1}\{b\}\neq \emptyset$ or
    $\Argmin H_{\bar\lambda}$ is a polyhedron, then
    $H$ also satisfies the KL property at $\bar{x}$ with exponent $\alpha$.
\end{theorem}
\begin{proof}
    First, because $\bar{\lambda}$ is a Lagrange multiplier, we have
    \begin{align}\label{ArgminF}
        H(\bar{x}) =\inf H = \inf H_{\bar{\lambda}} = H_{\bar{\lambda}}(\bar{x}) \ {\rm and}\
        \bar{x} \!\in\! \Argmin H \!=\!\Argmin H_{\bar{\lambda}} \cap G^{-1}\{b\},
    \end{align}
    where the last equality holds because of \cite[Theorem 28.1]{Rockafellar+1970}.
    Second, since $H_{\bar{\lambda}}$ satisfies the KL property with exponent $\alpha$ at $\bar{x}$, in view of \cite[Theorem~5]{bolte2017error}, there
    exist $\epsilon>0$,  $r_0>0$ and $\bar{c}>0$ such that for any $x$ satisfying $\|x - \bar{x}\rVert \leq \epsilon$ and
    $H_{\bar{\lambda}}(\bar{x})<H_{\bar{\lambda}}(x)< H_{\bar{\lambda}}(\bar{x})+ r_0$, we have
    \begin{align}\label{ArgminFineq}
        \dist(x,\Argmin H_{\bar{\lambda}}) \leq \bar{c}(H_{\bar{\lambda}}(x) - H_{\bar{\lambda}}(\bar{x}))^{1-\alpha}.
    \end{align}
    Now, for any $x$ satisfying $\|x - \bar{x}\rVert \leq \epsilon$ and
    $H(\bar{x})<H(x)< H(\bar{x})+ r_0$, we have from \eqref{ArgminF} that
    \begin{align}
        &\dist(x, \Argmin H) = \dist \left(x,\Argmin H_{\bar{\lambda}} \cap G^{-1}\{b\}\right)
        \overset{\mathop{(a)}}{\leq} \kappa \dist(x, \Argmin H_{\bar{\lambda}}) \nonumber \\
        & \overset{(b)}\leq \kappa \bar{c}\,(H_{\bar{\lambda}}(x) - H_{\bar{\lambda}}(\bar{x}))^{1-\alpha}
        \overset{(c)} = \kappa \bar{c}\,(H(x) - H(\bar{x}))^{1-\alpha}, \nonumber
    \end{align}
    where $(a)$ holds for some constant $\kappa > 0$ (independent of $x$) thanks to \cite[Corollary 3]{bauschke1999strong}, (b) holds because of the definition of $H_{\bar \lambda}$, the fact that $Gx = G\bar x = b$ and \eqref{ArgminFineq}, and (c) follows from the definition of $H_{\bar \lambda}$ and the fact that $Gx = G\bar x = b$. The conclusion concerning KL property now follows immediately upon invoking \cite[Theorem~5]{bolte2017error}.
\end{proof}

\section{Numerical experiments} \label{section_numerical experiments}
\add{In this section, we perform numerical experiments for $\Alg$ on instances of \eqref{reform_prob3} (or, equivalently, \eqref{reform_prob2}). For notational simplicity, we describe our implementation based on the problem formulation \eqref{reform_prob2}. Recall that we use bold-faced letters to denote vectors and matrices with complex entries.
By Remark~\ref{choose_delta} and Corollary~\ref{local_rate}, we can deduce the following asymptotic bounds  for the sequence $\{({\bm x}^t, {\bm Y}^t)\}$
generated by $\Alg$ with $\delta = \frac{1}{2}$:
\begin{align*}
   & \left|  \| \Pi_{\Omega} (w\circ({\bm x}^t - \bar{\bm x}))\|_1 - \theta \right| = {\cal O} \left( (t+1)^{-1/2} \right), \ \  \|{\cal H}(\bm{x}^t) - \bm{Y}^t\|_F  = {\cal O} \left( (t+1)^{-1/2} \right), \\
   & {\rm dist}((\bm{x}^t, \bm{Y}^t), \Argmin F) = {\cal O}\left( (t+1)^{-1/2^{m+n+q}} \right),
\end{align*}
where $\theta$ and $\Argmin F$ are the optimal value  and   the solution set of  \eqref{reform_prob2}, respectively, and $\|\cdot\|_F$ is the Frobenius norm.}

\add{We next discuss how $\Alg$ can be applied to solving \eqref{reform_prob2}. We will study the numerical performance of $\Alg$ on the instances of \eqref{reform_prob2}.
All the numerical tests are performed in MATLAB R2022b on a 64-bit PC with Intel(R) Core(TM) i7-10700 CPU @2.90GHz (16CPUs), 2.9GHz and 32GB of RAM.\footnote{\add{The codes for the numerical tests in this section can be founded in \href{https://github.com/zengliaoyuan/ProxCG\_HankelMatrixCompletion}{https://github.com/zengliaoyuan/ProxCG\_HankelMatrixCompletion}}}}

\add{{\bf Algorithm settings}: We apply $\Alg$ with $\delta = \frac{1}{2}$ and $\beta_0 = 0.3$ to \eqref{reform_prob2}. We let $ H_0=10^{-6} $, $ M_f=0 $ and $ \mu=1 $. Then each iteration of $\Alg$ applied to \eqref{reform_prob2} consists of the following updates:
\begin{equation*}
    \left\{\begin{array}{cl}
        \!\!\! {\bm x}^{t+1} \!\!\!&\!\!\!=\!\!\!\!\!\! \mathop{{\rm argmin}}\limits_{\|{\bm x} - \Pi_\Omega(\bar{\bm x})\|_2 \le \widetilde\sigma }\!\!\!\!\! \| \Pi_{\Omega} (w\!\circ\! ({\bm x} \!-\! \bar{{\bm x}}))\|_1\!\! +\!{\rm Re} \langle \beta_t {\cal H}^*({\cal H}({\bm x}^t)\! -\!{\bm Y}^t), {\bm x}\rangle\! +\! \frac{H_0 + \beta_t \lambda_H}{2}\|{\bm x} \!-\!{\bm x}^t\|_2^2,\!\!\!  \\
        \!\!\! {\bm U}^t\!\!\!& \!\!\!\in \mathop{{\rm Argmin}}\limits_{\|{\bm Y}\|_* \le \sigma}~{\rm Re} \langle {\bm Y}^t - {\cal H}({\bm x}^{t+1}), {\bm  Y} -{\bm Y}^t \rangle, \\  
        \!\!\! {\bm Y}^{t+1}\!\!\!& \!\!\!= {\bm Y}^t + \alpha_t({\bm U}^t - {\bm Y}^t),
    \end{array}\right.
\end{equation*}
 where $\widetilde\sigma =\sigma + \|\Pi_\Omega({\bar{\bm x}})\|_2 + 1$, $\lambda_H = \min \{m,q\}$, and ${\rm Re}({\bm a})$ denotes the real part of a complex number ${\bm a}$.  }

\add{We now describe how to solve the two subproblems presented above.
For the ${\bm x}$-update, we first compute:
\[
    \tilde{{\bm x}}_j^{t+1}  :=  \begin{cases}
                                             \bar{{\bm x}}_j+ {\rm arg} ({\bm c}_j)\max \left\{ |{\bm c}_j| -\frac{w_j}{H_0 + \beta_t \lambda_H}, 0 \right\} & \text{if}\  j \in \Omega, \\
                                            {\bm x}_j^t -\frac{\beta_t}{H_0 + \beta_t\lambda_H} \left[{\cal H^*H}({\bm x}^t)  -{\cal H}^*({\bm Y}^t)\right]_j & \text{if}\  j \notin \Omega,\\
                                         \end{cases}
\]
where ${\bm c}_j = {\bm x}_j^t - \bar{{\bm x}}_j -\frac{\beta_t}{H_0 + \beta_t\lambda_H} \left[{\cal H^*}({\cal H}({\bm x}^t)-{\bm Y}^t)\right]_j$. Then we have
\[
{\bm x}^{t+1} =
\begin{cases}
\tilde{{\bm x}}^{t+1} &
{\rm if\ } \|\tilde{{\bm x}}^{t+1} - \Pi_\Omega(\bar{\bm x})\|_2 \le \widetilde\sigma,\\
\Pi_\Omega(\bar{\bm x}) + \widetilde \sigma \frac{\tilde{{\bm x}}^{t+1}  - \Pi_\Omega(\bar{\bm x})}{\| \tilde{{\bm x}}^{t+1}  - \Pi_\Omega(\bar{\bm x})\|_2}& {\rm otherwise}.
\end{cases}
\]
For the other subproblem, we have
${\bm U}^t  =  -\sigma {\bm u}_{\max}{\bm v}_{\max}^*$,
where ${\bm u}_{\max}$ and ${\bm v}_{\max}$ are left and right singular vectors corresponding to the top singular value of ${\bm Y}^t - {\cal H}({\bm x}^{t+1})$, respectively. More importantly, in our implementation, we {\em do not} form ${\bm Y}^t$ explicitly but maintain its thin SVD triple, and leverage the fact that ${\bm U}^t$ is of rank-one and the rank-one SVD update technique proposed in \cite{brand2006fast} to update this triple; we also take advantage of this triple and the special structure of ${\cal H}({\bm x}^{t+1})$ to compute ${\bm u}_{\max}$ and ${\bm v}_{\max}$ via the MATLAB command {\sf svds}, and take advantage the triple to compute $ {\cal H}^*({\bm Y}^t) $ via fast convolutions in a way similar to \cite[Section~2]{cai2023structured}.}

\add{We initialize $\Alg$ at $({\bm x}^0, {\bm Y}^0) = (\Pi_{\Omega} (\bar{\bm x}), {\cal H}(\Pi_{\Omega} (\bar{\bm x})) )$.\footnote{\add{Here, we use $ {\bm Y}^0 = {\cal H}(\Pi_{\Omega} (\bar{\bm x})) $ the same as that in ADMM below. We do not need to form ${\bm Y}^0$ explicitly in our code because we can deduce that $ {\bm x}^1={\bm x}^0 $ and $ {\bm Y}^1 $ is a zero matrix. Then the SVD triple of $ {\bm Y}^1 $ can be obtained directly.} }
We terminate $\Alg$ once  $t > 50000$. 
}

\add{As a benchmark, we also apply the ADMM to solve \eqref{reform_prob2}, whose iterates are
\begin{equation}\label{ADMM}\left\{
    \begin{array}{cl}
       {\bm x}^{t+1}  \!\!\!\!& =  \mathop{{\rm argmin}}\limits_{\bm x \in \mathbb{C}^n }~ \| \Pi_{\Omega} (w\!\circ\! ({\bm x}\! -\! \bar{{\bm x}}))\|_1 \!+\!{\rm Re} \left\langle {\cal H}^*({\bm \Lambda}^t), {\bm x}\right\rangle + \frac\beta{2}\|{\cal H}({\bm x}) - {\bm Y}^t\|_F^2,  \\[3 pt]
       {\bm Y}^{t+1}  \!\!\!\!& = \mathop{{\rm argmin}}\limits_{\|\bm Y\|_* \le \sigma}~\frac{\beta}{2}\| {\bm Y} - {\cal H}({\bm x}^{t+1}) - \frac{1}{\beta}{\bm \Lambda}^t\|_F^2, \\[3 pt]
       {\bm \Lambda}^{t+1} \!\!\!\!& = {\bm \Lambda}^t + \beta( {\cal H}({\bm x}^{t+1}) - {\bm Y}^{t+1}),
    \end{array} \right.
\end{equation}
where $\beta > 0$. Notice that the ${\bm x}$-update admits a closed form solution, while the ${\bm Y}$-update involves projections onto the nuclear norm ball: the latter necessitates forming the matrix ${\cal H}({\bm x}^{t+1}) + {\bm \Lambda}^t/\beta$ and performing a full SVD, which can be inefficient or prohibitively expensive when the matrix size is huge.
}

\add{We choose $\beta =1$ and initialize ADMM at $({\bm Y}^0, {\bm \Lambda}^0) = \left({\cal H}(\Pi_{\Omega} (\bar{\bm x})),0\right)$. To describe the termination criteria, we first note that the dual problem of \eqref{reform_prob2} is
\begin{equation}\label{dualprob}
  \begin{array}{cl}
    \max\limits_{{\bm \Lambda}\in\mathbb{C}^{m \times q}} & {\rm Re} \langle {\cal H}(\Pi_\Omega(\bar{\bm x})), {\bm \Lambda} \rangle- \sigma\|{\bm \Lambda}\|_2 \\
    {\rm s.t.} & {\cal H}^*({\bm \Lambda})_j = 0\ \text{if }j \notin \Omega,~|{\cal H }^*({\bm \Lambda})_j|\le w_j \ \text{if }j \in \Omega,
  \end{array}
\end{equation}
where $\|{\bm \Lambda}\|_2$ is the largest singular value of ${\bm \Lambda}$.
Since the $\{({\bm x}^t,{\bm Y}^t)\}$ from \eqref{ADMM} converges to a solution of \eqref{reform_prob2} and $\{{\bm \Lambda}^t\}$ converges to a solution of \eqref{dualprob}, we terminate the ADMM when the relative gap (${\rm gap}_{\rm r}(t)$) and relative dual feasibility violation (${\rm feas}_{\rm r}(t)$) defined below are small:\footnote{\add{We can compute $\|{\bm \Lambda}^t\|_2$ directly from the singular values of  ${\bm Y}^{t}$ and  $\mathcal{H}({\bm x}^{t})+\frac{1}{\beta}{\bm \Lambda}^{t-1}$ because the ${\bm Y}$-update suggests that ${\bm Y}^{t}$ and  $\beta\mathcal{H}({\bm x}^{t})+{\bm \Lambda}^{t-1}$ have simultaneous SVD.} }
\[
    {\rm gap}_{\rm r}(t) = \frac{\left| \| \Pi_{\Omega} (w\circ ({\bm x}^{t} - \bar{{\bm x}}))\|_1 - {\rm Re}\langle {\cal H}(\Pi_\Omega(\bar {\bm x})), {\bm \Lambda}^{t} \rangle+ \sigma\|{\bm \Lambda}^{t}\|_2 \right|}{\max\{1,\| \Pi_{\Omega} (w\circ ({\bm x}^{t} - \bar{{\bm x}}))\|_1\}},
\]
\[
    {\rm feas}_{\rm r}(t) =\frac{ \sum_{j \notin \Omega}|{\cal H }^*({\bm \Lambda}^t)_j|+ \sum_{j \in \Omega} \max\left( |{\cal H }^*({\bm \Lambda}^t)_j| - w_j, 0 \right)}{\max\{1,\|{\bm \Lambda}^t\|_2\}}.
\]
Specifically, we terminate ADMM when $t>10^4$  or $\max\{{\rm gap}_{\rm r}(t),2\cdot{\rm feas}_{\rm r}(t)\} < 0.1$. We are not using primal feasibility violation as a termination criterion as we will report the approximate primal feasibility violation at termination; see {\bf rel}$_{\rm feas}$ in Table~\ref{tab:ProxADMM} below.}

\add{{\bf Data generation}:
We generate a spectrally sparse $\bm{ox}\in \mathbb{C}^n$ with exactly $r$ active frequencies following the procedure in
\cite[Section III.A]{cai2021accelerated} with frequency setting (b). We set $\sigma = 0.97\times\|{\cal H}(\bm{ox})\|_*$ and we generate $\bar{\bm x}$ by adding Laplacian noise
with mean 0 and variance $10^{-4}$ to the real and imaginary parts of $\bm{ox}$. Finally, we uniformly sample $[\alpha \times n]$ entries from $\bar{\bm x}$ and record the indices of those entries as $\Omega$.}

\add{{\bf Numerical result}: We consider $r = 7$ and $n = 2^{j}$ with $j\in \{10,12,14, 16\}$. For each problem size, we fix $\alpha$ at 0.4 and generate 10 random test instances as described above.
Our computational results comparing the performance of  $\Alg$ and ADMM, averaged over the $10$ random instances, are presented in Table~\ref{tab:ProxADMM}.}
\add{Here, \textbf{size} represents the dimension of $\bm{ox}$,
\textbf{err}$:=\| \bm{x}_{\rm out} - \bm{ox}\|_2 / \| \bm{ox}\|_2$, where ${\bm x}_{\rm out}$ is the last iterate returned by the algorithm,  \textbf{obj} stands for the objective value at ${\bm x}_{\rm out}$, \textbf{iter}  stands for the number of iterations,
\textbf{cpu} is the CPU time, $\textbf{rel}.\sigma_r$ and $\textbf{rel}.\sigma_{r+1}$ are defined as $\sigma_r / \sigma_1$, $\sigma_{r+1}/ \sigma_1$, where $\sigma_1$, $\sigma_r$, $\sigma_{r+1}$ are the
largest, $r$-th largest and $(r+1)$-th largest singular values of ${\cal H}(\bm x_{\rm out})$, respectively, and \textbf{rel}$_{\rm feas}:= \sum_{j=1}^{r+1}\sigma_j / \sigma -1$.\footnote{\add{We do not use $\| {\cal H}(\bm x_{\rm out})\|_*$ in \textbf{rel}$_{\rm feas}$ because it requires the full set of singular values of $ {\cal H}(\bm x_{\rm out}) $, and is prohibitively expensive to compute in MATLAB for large $ n $.}} \add{In Table~\ref{tab:ProxADMM}, there is no result of ADMM with $ n=2^{16} $ because the computer runs out of memory when performing the full SVD in the updating of the variable $ Y $ in \eqref{ADMM}.}
}

\begin{table}[htbp]
  \centering
  \caption{\add{Comparing $\Alg$ and ADMM on solving low rank Hankel matrix completion problems.}}
  \label{tab:ProxADMM}
  \add{
  \scalebox{0.75}{
  \begin{tabular}{ccccccccc}
  \toprule
 \textbf{method}&\textbf{size} & \textbf{err} & \textbf{obj} & \textbf{iter} & \textbf{cpu} &\textbf{rel}.$\sigma_r$ &\textbf{rel}.$\sigma_{r+1}$ & \textbf{rel}$_{\rm feas}$ \\
  \midrule
    $\Alg$&     $ 2^{10} $ &   0.0332  & 2.01e+04  & 50000 &  1279.34 & 5.27e-01 & 2.26e-04 & 6.92e-04 \\
    ADMM  &     $ 2^{10} $ &   0.0337  & 2.16e+04  &   137.20 &    70.47 & 5.26e-01 & 1.41e-05 & 2.81e-06 \\
    $\Alg$  &   $ 2^{12} $ &   0.0314  & 3.14e+05  & 50000 &  4514.02 & 5.86e-01 & 1.33e-04 & 6.70e-04 \\
    ADMM  &     $ 2^{12} $ &   0.0320  & 3.38e+05  &   172.70 &  2059.07 & 5.85e-01 & 3.50e-06 & 6.43e-07 \\
   $\Alg$  &    $ 2^{14} $ &   0.0298  & 4.94e+06  & 50000 & 16387.54 & 5.81e-01 & 8.89e-05 & 6.64e-04 \\
    ADMM  &     $ 2^{14} $ &   0.0304  & 5.32e+06  &   248.70 & 95554.93 & 5.80e-01 & 5.36e-07 & 9.36e-08 \\
    $\Alg$  &   $ 2^{16} $ &   0.0291  & 7.73e+07  & 50000 & 63993.27 & 5.72e-01 & 6.94e-05 & 6.94e-04 \\
  \bottomrule
\end{tabular}}
}
\end{table}

\appendix

\section{\add{An auxiliary lemma}}

\begin{color}{blue}
This lemma establishes an upper bound on an auxiliary quadratic penalty function along the sequence generated by $\Alg$.
\begin{lemma} \label{theorem_f+beta_h}
    Consider \eqref{problem}.
    Suppose that Assumptions~\ref{Assumption 1}, \ref{Assumption 2} hold and $(x^*,y^*)$ solves \eqref{problem}. Let
    $\{(x^t,y^t)\}$ be generated by $\Alg$. Then the following statements hold.
    \begin{enumerate} [{\rm (i)}]
      \item If $\mu \in (0,1)$ in Assumption \ref{Assumption 2}, then for all $t \ge 2$,
    \begin{align}
        &f(x^{t}) + g(y^{t}) +\frac{\beta_{t-1}}{2}\|Ax^t+By^t-c\rVert^2 - f(x^*) -g(y^*) \nonumber \\
        &\leq  \frac{\omega_1}{t(t+1)}+\frac{\omega_2}{(t+1)^{1-\delta}}+\frac{\omega_3}{(t+1)^{\nu}}+ \frac{\omega_4}{(t+1)^{\mu}},\nonumber
    \end{align}
    where $\omega_1$, $\omega_2$, $\omega_3$ and $\omega_4$ are defined in \eqref{w1w2w3} and \eqref{w3w4}.
    \item If $\mu =1$ in Assumption \ref{Assumption 2}, then for all $t \ge 2$,
    \begin{align}
      &f(x^{t}) + g(y^{t}) +\frac{\beta_{t-1}}{2}\|Ax^t+By^t-c\rVert^2 - f(x^*) -g(y^*) \nonumber \\
      &\leq   \frac{\omega_1}{t(t+1)}+\frac{\omega_2}{(t+1)^{1-\delta}}+\frac{\omega_3}{(t+1)^{\nu}} +\frac{\omega_5}{t+1}, \nonumber
    \end{align}
    where $\omega_1 $, $\omega_2$, $\omega_3$ and $\omega_5$ are defined in \eqref{w1w2w3} and \eqref{w3w4}.
    \end{enumerate}
\end{lemma}
\begin{proof}
    Define $h(x,y) =\frac{1}{2}\|Ax+By-c\rVert^2$.
    Then  we  have for all $t \ge 0$ that
    \begin{align}
        &f(x^{t+1})+\beta_th(x^{t+1},y^t) +\langle \nabla f_1(x^t), x^{t+1}-x^t\rangle +\frac{H_t}{2}\|x^{t+1}-x^t\rVert^2 \nonumber \\
        &\overset{\mathop{(a)}}{\leq}  f_1(x^{t+1})+f_2(x^{t+1}) +\langle \nabla f_1(x^t), x^{t+1}-x^t\rangle +\frac{H_t}{2}\|x^{t+1}-x^t\rVert^2
         \nonumber \\
        &\ \ \ +\beta_t\langle A^*(Ax^t+By^t-c), x^{t+1}-x^t\rangle +\beta_th(x^t,y^t) +\frac{\beta_t\lambda_A}{2}\| x^{t+1}-x^t\rVert^2 \nonumber\\
        &\overset{\mathop{(b)}}{\leq} f_1(x^{t+1})+ f_2(x^t+\alpha_t(x^*-x^t)) +\alpha_t\langle \nabla f_1(x^t), x^*-x^t\rangle  \nonumber \\
        &\ \ \ +\alpha_t\beta_t\langle A^*(Ax^t+By^t-c), x^*-x^t\rangle + \beta_th(x^t,y^t) +
        \frac{H_t+\beta_t\lambda_A}{2}\alpha_t^2\|x^*-x^t\rVert^2 \nonumber \\
        &\overset{\mathop{(c)}}{\leq} f_1(x^{t+1})+ (1-\alpha_t)f_2(x^t)+\alpha_tf_2(x^*) +\alpha_t\langle \nabla f_1(x^t), x^*-x^t\rangle  \nonumber \\
        &\ \ \ +\alpha_t\beta_t\langle A^*(Ax^t+By^t-c), x^*-x^t\rangle + \beta_th(x^t,y^t) +
        \frac{H_t+\beta_t\lambda_A}{2}\alpha_t^2\|x^*-x^t\rVert^2, \nonumber
    \end{align}
    where $(a)$ holds because $h(\cdot,y)$ has Lipschitz continuous gradient with Lipschitz constant
    $\lambda_A$, $(b)$ holds because of \eqref{get_xk}
    and the fact that $x^t+\alpha_t(x^*-x^t) \in \mathrm{dom} f$ and $(c)$ holds because $f_2$ is convex.

    Define $\mathcal{L}_{f_1}(x^{t+1}, x^t) = f_1(x^{t+1})-f_1(x^t)-\langle \nabla f_1(x^t),x^{t+1}-x^t\rangle$ and rearrange terms in the above display, we have upon invoking the definition of $R^t$ in \eqref{Rt} that
    \begin{align}
        &f(x^{t+1}) +\beta_th(x^{t+1},y^t) \nonumber \\
        &\leq f_1(x^t)+\mathcal{L}_{f_1}(x^{t+1}, x^t) +(1-\alpha_t)f_2(x^t)+\alpha_tf_2(x^*)+\alpha_t\langle \nabla f_1(x^t), x^*-x^t\rangle  \nonumber \\
        &\ \ \ +\!\frac{H_t+\beta_t\lambda_A}{2}\alpha_t^2\|x^*\!\!-\!x^t\rVert^2\!+\!\alpha_t\beta_t\langle A^*R^t, x^*\!\!-\!x^t\rangle \!+\! \beta_th(x^t,y^t) \!-\!\frac{H_t}{2}\|x^{t+1}\!\!-\!x^t\rVert^2\!\!\!\!\!\! \nonumber \\
        &\overset{\mathop{(a)}}{\leq}  (1\!-\!\alpha_t)f_1(x^t)\!+\!(1\!-\!\alpha_t)f_2(x^t) \!+\!\alpha_t f_1(x^*)\!+\!\alpha_tf_2(x^*)\!+\!\frac{H_t\!+\!\beta_t\lambda_A}{2}\alpha_t^2\|x^*\!-\!x^t\rVert^2\!\!\!\!\!\! \nonumber \\
        &\ \ \ +\alpha_t\beta_t\langle A^*R^t, x^*-x^t\rangle + \beta_th(x^t,y^t) -\frac{H_t}{2}\|x^{t+1}-x^t\rVert^2+\mathcal{L}_{f_1}(x^{t+1}, x^t)\!\!\!\!  \nonumber \\
        &\leq (1\!-\!\alpha_t)(f_1(x^t)\!+\!f_2(x^t)) +\alpha_t f_1(x^*)+\alpha_tf_2(x^*)+\frac{H_t+\beta_t\lambda_A}{2}\alpha_t^2\|x^*-x^t\rVert^2 \nonumber \\
        &\ \ \ +\!\alpha_t\beta_t\langle A^*R^t, x^*\!-\!x^t\rangle + \beta_th(x^t,y^t) \!-\!\frac{H_t}{2}\|x^{t+1}\!-\!x^t\rVert^2 \!+\!\frac{M_f}{\mu+1}\|x^{t+1}-x^t\rVert^{\mu+1},\!\!\!\!\!\!\!  \label{f_beta_h}
    \end{align}
    where $(a)$ holds because $f_1$ is convex and the last inequality follows from \eqref{Holder_property}.

    Define
    \begin{align}\label{Definition_zeta}
      \zeta_{t} = -\frac{H_t}{2}\|x^{t+1}-x^t\rVert^2 +\frac{M_f}{\mu+1}\|x^{t+1}-x^t\rVert^{\mu+1}
    \end{align}
    for notational simplicity. Then, by rearranging  terms in \eqref{f_beta_h} and recalling the definition of $R^t$ in \eqref{Rt}, we obtain that for all $t\ge 0$,
     \begin{align}
        &f(x^{t+1}) +\beta_th(x^{t+1},y^t) -f(x^*)\nonumber\\
        &\leq (1-\alpha_t)(f(x^t)-f(x^*))+\frac{H_t+\beta_t\lambda_A}{2}\alpha_t^2\|x^*-x^t\rVert^2  \nonumber \\
        &\ \ \ +\alpha_t\beta_t\langle A^*(Ax^t+By^t-c), x^*-x^t\rangle+ \beta_th(x^t,y^t) +\zeta_t.  \label{f+beta_h-f*_zeta}
    \end{align}

     Next, we deduce an analogous relation involving $g$. To this end, notice that for all $t\ge 0$,
    \begin{align}
        &g(y^{t+1}) + \beta_th(x^{t+1},y^{t+1}) \nonumber \\
        &\overset{\mathop{(a)}}{\leq} g_1(y^{t})+ \langle \nabla g_1(y^t)+\beta_tB^*\widetilde R^t, y^{t+1}-y^t\rangle +
        \frac{M_g}{\nu+1}\|y^{t+1}-y^t\rVert^{\nu+1} \nonumber\\
        &\ \ \ +g_2(y^{t+1}) +\beta_th(x^{t+1},y^t)+\frac{\beta_t\lambda_B}{2}\|y^{t+1}-y^t\rVert^2 \nonumber \\
        &\overset{\mathop{(b)}}{\leq} g_1(y^{t})+ \alpha_t\langle \nabla g_1(y^t)+\beta_tB^*\widetilde R^t, u^t-y^t\rangle +
        \frac{M_g}{\nu+1}\alpha_t^{\nu+1}\|u^t-y^t\rVert^{\nu+1} \nonumber\\
        &\ \ \ +\alpha_tg_2(u^t)+(1-\alpha_t)g_2(y^{t}) +\beta_th(x^{t+1},y^t)+\frac{\beta_t\lambda_B}{2}\alpha_t^2\|u^t-y^t\rVert^2\nonumber \\
        &\overset{\mathop{(c)}}{\leq} g_1(y^{t})+ \alpha_t\langle \nabla g_1(y^t)+\beta_tB^*\widetilde R^t, y^*-y^t\rangle +
        \frac{M_g}{\nu+1}\alpha_t^{\nu+1}\|u^t-y^t\rVert^{\nu+1} \nonumber\\
        &\ \ \ +\alpha_tg_2(y^*)+(1-\alpha_t)g_2(y^{t}) +\beta_th(x^{t+1},y^t)+\frac{\beta_t\lambda_B}{2}\alpha_t^2\|u^t-y^t\rVert^2 \nonumber \\
        &\overset{\mathop{(d)}}{\leq} g_1(y^{t})+\alpha_tg_1(y^*)-\alpha_tg_1(y^t)+ \alpha_t\beta_t\langle B^*\widetilde R^t, y^*-y^t\rangle +
        \frac{M_g}{\nu+1}\alpha_t^{\nu+1}\|u^t-y^t\rVert^{\nu+1} \nonumber\\
        &\ \ \ +\frac{\beta_t\lambda_B}{2}\alpha_t^2\|u^t-y^t\rVert^2+\alpha_tg_2(y^*)+(1-\alpha_t)g_2(y^{t}) +\beta_th(x^{t+1},y^t) \nonumber \\
        &\overset{\mathop{(e)}}{\leq} (1-\alpha_t)g(y^{t})+\alpha_tg(y^*)+ \alpha_t\beta_t\langle B^*\widetilde R^t, y^*-y^t\rangle  \nonumber\\
        &\ \ \ +\beta_th(x^{t+1},y^t)+
        \frac{M_g}{\nu+1}\alpha_t^{\nu+1}D_g^{\nu+1}+\frac{\beta_t\lambda_B}{2}\alpha_t^2D_g^2 ,  \nonumber
    \end{align}
    where we used the definition of $\widetilde R^t$ in \eqref{tRt} and $(a)$ holds thanks to \eqref{Holder_property} and the fact that $h(x,\cdot)$ has Lipschitz continuous gradient with Lipschitz constant $\lambda_B$,
    $(b)$ holds because of the convexity of $g_2$ as well as the definition of $y^{t+1}$ in \eqref{get_yk}, $(c)$ holds due to \eqref{get_uk},
    $(d)$ holds since $g_1$ is convex and $(e)$ holds because of the definition of $D_g$.

    Rearranging terms in the above inequality, we obtain upon recalling the definition of $\widetilde R^t$ in \eqref{tRt} that
    \begin{align}
        &g(y^{t+1}) + \beta_th(x^{t+1},y^{t+1}) - g(y^*)\nonumber\\
        &\leq (1-\alpha_t)(g(y^t)-g(y^*))+ \alpha_t\beta_t\langle B^*(Ax^{t+1}+By^t-c), y^*-y^t\rangle  \nonumber \\
        &\ \ \ + \beta_th(x^{t+1},y^t)+
        \frac{M_g}{\nu+1}\alpha_t^{\nu+1}D_g^{\nu+1}+\frac{\beta_t\lambda_B}{2}\alpha_t^2D_g^2.   \label{g+beta_h-g*}
    \end{align}
    Summing \eqref{f+beta_h-f*_zeta} and \eqref{g+beta_h-g*}, we have upon rearranging terms that for all $t\ge 1$,
    \begin{align}
        &f(x^{t+1})+ g(y^{t+1})+ \beta_{t}h(x^{t+1},y^{t+1})-f(x^*)-g(y^*)  \nonumber \\
        &\leq (1-\alpha_t)\left(f(x^t)+g(y^t)-f(x^*)-g(y^*)\right) + \beta_th(x^t,y^t)\nonumber \\
        &\ \ \ +\alpha_t\beta_t\langle A^*(Ax^t+By^t-c), x^*-x^t\rangle + \alpha_t\beta_t\langle B^*(Ax^{t+1}+By^t-c), y^*-y^t\rangle \nonumber \\
        &\ \ \ +\frac{M_g}{\nu+1}\alpha_t^{\nu+1}D_g^{\nu+1}+\frac{\beta_t\lambda_B}{2}\alpha_t^2D_g^2+ \frac{H_t+\beta_t\lambda_A}{2}\alpha_t^2
        \|x^*-x^t\rVert^2+\zeta_t. \nonumber\\
        & = (1-\alpha_t)\left(f(x^t)+g(y^t)+\beta_{t-1}h(x^t,y^t)-f(x^*)-g(y^*)\right)\nonumber \\
        &\ \ \ + (1-\alpha_t)(\beta_t-\beta_{t-1})h(x^t,y^t) +\alpha_t\beta_th(x^t,y^t) \nonumber \\
        &\ \ \ +\alpha_t\beta_t\langle A^*(Ax^t+By^t-c), x^*-x^t\rangle + \alpha_t\beta_t\langle B^*(Ax^{t+1}+By^t-c), y^*-y^t\rangle \nonumber \\
        &\ \ \ +\frac{M_g}{\nu+1}\alpha_t^{\nu+1}D_g^{\nu+1}+\frac{\beta_t\lambda_B}{2}\alpha_t^2D_g^2+ \frac{H_t+\beta_t\lambda_A}{2}\alpha_t^2
        \|x^*-x^t\rVert^2+\zeta_t. \nonumber
    \end{align}
    Using this and applying Lemma~\ref{lemma_of_h} with $x_1 = x^t $, $x_2 = x^{t+1}$ and $y = y^t$, we have
    \begin{align}
        &f(x^{t+1})+ g(y^{t+1})+ \beta_{t}h(x^{t+1},y^{t+1})-f(x^*)-g(y^*)  \nonumber \\
        &\leq (1-\alpha_t)\left(f(x^t)+g(y^t)+\beta_{t-1}h(x^t,y^t)-f(x^*)-g(y^*)\right)  \nonumber\\
        &\ \ \ +\alpha_t\beta_t\langle Ax^t-Ax^{t+1}, By^t-By^* \rangle +\frac{M_g}{\nu+1}\alpha_t^{\nu+1}D_g^{\nu+1}+\frac{\beta_t\lambda_B}{2}\alpha_t^2D_g^2  \nonumber \\
        &\ \ \
        +\frac{H_t+\beta_t\lambda_A}{2}\alpha_t^2\|x^*-x^t\rVert^2+\zeta_t.     \nonumber
    \end{align}

    Next, multiplying $(t+1)(t+2)$ to  both sides of the above inequality and rearranging terms, we obtain upon noting $(t+1)(t+2)(1-\alpha_t) = t(t+1)$ and letting
    $\Upsilon_t = t(t+1)(f(x^t)+g(y^t)+\beta_{t-1}h(x^t,y^t)-f(x^*)-g(y^*))$ that for all $t \geq 1$,
    \begin{align}
        &\Upsilon_{t+1} - \Upsilon_t\nonumber\\
        &\leq(t+1)(t+2)\alpha_t\beta_t\langle Ax^t-Ax^{t+1}, By^t-By^*\rangle + (t+1)(t+2)\frac{M_g}{\nu+1}\alpha_t^{\nu+1}D_g^{\nu+1} \nonumber \\
        &\ \ \ +(t+1)(t+2)\frac{\beta_t\lambda_B}{2}\alpha_t^2D_g^2 + (t+1)(t+2)\frac{H_t+\beta_t\lambda_A}{2}\alpha_t^2D_f^2 +(t+1)(t+2) \zeta_t\nonumber \\
        &=(t+1)(t+2)\alpha_t\beta_t\langle Ax^t-Ax^{t+1}, By^t-By^*\rangle+\frac{2^{\nu+1}}{\nu+1}\frac{t+1}{(t+2)^{\nu}}M_gD_g^{\nu+1}  \nonumber \\
        &\ \ \ +\frac{2(t+1)}{t+2}\beta_t\lambda_BD_g^2 +2(H_t+\beta_t\lambda_A)\frac{t+1}{t+2}D_f^2+(t+1)(t+2) \zeta_t \nonumber \\
        &\leq (t+1)(t+2)\alpha_t\beta_t\langle Ax^t-Ax^{t+1}, By^t-By^*\rangle +\frac{2^{\nu+1}}{\nu+1}(t+1)^{1-\nu}M_gD_g^{\nu+1}  \nonumber \\
        &\ \ \ +2\beta_t\lambda_BD_g^2 +2(H_t+\beta_t\lambda_A)D_f^2+(t+1)(t+2) \zeta_t. \nonumber
    \end{align}

    The above inequality further implies that for all $t\ge 2$,
    \begin{align}
        &\Upsilon_t = \sum_{k=1}^{t-1} (\Upsilon_{k+1} - \Upsilon_k) +\Upsilon_1 \nonumber \\
        &\leq \sum\limits_{k=1}^{t-1} \Big[(k+1)(k+2)\alpha_k\beta_k\langle Ax^k-Ax^{k+1}, By^k-By^*\rangle +(k+1)(k+2) \zeta_k  \nonumber \\
        &\ \ \ \left.+\frac{2^{\nu+1}}{\nu+1}(k+1)^{1-\nu}M_gD_g^{\nu+1} +2\beta_k\lambda_BD_g^2 +2(H_k+\beta_k\lambda_A)D_f^2  \right] + \Upsilon_1 \nonumber  \\
        &\overset{\mathop{(a)}}{\leq} \left(2^{\delta +3}\beta_0+\frac{(16+8\delta)\beta_0}{1+\delta}(t+1)^{1+\delta}\right)D_2 +\sum\limits_{k=1}^{t-1}(k+1)(k+2)\zeta_k \nonumber \\
        &\ \ \ +\sum\limits_{k=1}^{t-1}\frac{2^{\nu+1}}{\nu+1}(k+1)^{1-\nu}M_gD_g^{\nu+1} +2(\lambda_AD_f^2+ \lambda_BD_g^2)\sum\limits_{k=1}^{t-1} \beta_k
        +2D_f^2\sum\limits_{k=1}^{t-1}H_k   +\Upsilon_1    \nonumber \\
        &\overset{\mathop{(b)}}{\leq} \left(2^{\delta +3}\beta_0+\frac{(16+8\delta)\beta_0}{1+\delta}(t+1)^{1+\delta} \right)D_2+\Upsilon_1+\sum\limits_{k=1}^{t-1}(k+1)(k+2)\zeta_k  \nonumber \\
        &\ \ \ + \frac{2^{\nu+1}}{\nu+1}M_gD_g^{\nu+1}(t-1)t^{1 -\nu} +2(\lambda_AD_f^2+ \lambda_BD_g^2)(t-1)\beta_{t-1} +2D_f^2(t-1)H_{t-1}, \nonumber
    \end{align}
    where $(a)$ holds because of Lemma \ref{lemma_of_sum_Axk-Axk+1}, $(b)$ holds because
    $(k+1)^{1-\nu} \leq t^{1-\nu}$, $\beta_k \leq \beta_{t-1}$, $H_k\leq H_{t-1}$ for all $0\leq k\leq t-1$.
    Recall that $\Upsilon_t = t(t+1)(f(x^t)+g(y^t)+\beta_{t-1}h(x^t,y^t)-f(x^*)-g(y^*))$ and $\beta_t = \beta_0 (t+1)^\delta$. Then we have that for all $t \ge 2$,
    \begin{align}
        &f(x^t)+g(y^t)+\beta_{t-1}h(x^t,y^t)-f(x^*)-g(y^*) \nonumber \\
        &\leq \frac{2^{\delta +3}\beta_0D_2+ \Upsilon_1}{t(t+1)}+\frac{(16+8\delta)\beta_0D_2}{1+\delta}\frac{(t+1)^{\delta}}{t}
        +\frac{2^{\nu+1}}{\nu+1}M_gD_g^{\nu+1}\frac{t-1}{t}\frac{t^{1-\nu}}{t+1}
         \nonumber \\
        &\ \ \ +2(\lambda_BD_g^2+\lambda_AD_f^2)\beta_0\frac{t-1}{t}\frac{t^{\delta}}{t+1}\!+\!2D_f^2\frac{t-1}{t}\frac{H_{t-1}}{t+1}+ \sum\limits_{k=1}^{t-1}\frac{(k+1)(k+2)}{t(t+1)}\zeta_k  \nonumber \\
        &\overset{(a)}{\leq} \frac{2^{\delta +3}\beta_0D_2+ \Upsilon_1}{t(t+1)}+\frac{(16+8\delta)\beta_0D_2}{1+\delta}\frac{(t+1)^{\delta}}{t}
        +\frac{2^{\nu+1}}{\nu+1}M_gD_g^{\nu+1}\frac{t^{1-\nu}}{t+1}
         \nonumber \\
        &\ \ \ +2(\lambda_BD_g^2+\lambda_AD_f^2)\beta_0\frac{t^{\delta}}{t+1}+2D_f^2\frac{H_{t-1}}{t+1}+ \sum\limits_{k=1}^{t-1}\frac{(k+1)(k+2)}{t(t+1)}\zeta_k  \nonumber \\
        &\leq  \frac{2^{\delta +3}\beta_0D_2+ \Upsilon_1}{t(t+1)}\!+\!\frac{(16+8\delta)\beta_0D_2}{1+\delta}\frac{(t+1)^{\delta}}{t+1}\frac{t+1}{t}
        \!+\!\frac{2^{\nu+1}}{\nu+1}M_gD_g^{\nu+1}\frac{(t+1)^{1-\nu}}{t+1}
         \nonumber \\
        &\ \ \ +2(\lambda_BD_g^2+\lambda_AD_f^2)\beta_0\frac{(t+1)^{\delta}}{t+1}+2D_f^2\frac{H_{t-1}}{t+1}+ \sum\limits_{k=1}^{t-1}\frac{(k+1)(k+2)}{t(t+1)}\zeta_k   \nonumber \\
        &\overset{(b)}\leq  \frac{2^{\delta +3}\beta_0D_2+ \Upsilon_1}{t(t+1)}+\left(\frac{(32+16\delta)\beta_0D_2}{1+\delta}+2(\lambda_BD_g^2+\lambda_AD_f^2)\beta_0\right)\frac{1}{(t+1)^{1-\delta}}\nonumber \\
        &\ \ \ +\frac{2^{\nu+1}}{\nu+1}M_gD_g^{\nu+1}\frac{1}{(t+1)^{\nu}}+2D_f^2\frac{H_{t}}{t+1}+ \sum\limits_{k=1}^{t-1}\frac{(k+1)(k+2)}{t(t+1)}\zeta_k, \nonumber   \\
        & = \frac{\omega_1}{t(t+1)} + \frac{\omega_2}{(t+1)^{1-\delta}}+ \frac{\omega_3}{(t+1)^{\nu}}+2D_f^2\frac{H_{t}}{t+1}+ \sum\limits_{k=1}^{t-1}\frac{(k+1)(k+2)}{t(t+1)}\zeta_k \label{Gamma_t},
    \end{align}
    where $(a)$ holds because $\frac{t-1}{t} \leq 1$ for all $t \geq 1$, (b) holds because $H_t\ge H_{t-1}$ and $\frac{t+1}{t} \leq 2$ for all $t \geq 2$, and the last equality holds because of the definitions of $\omega_1$, $\omega_2$, $\omega_3$ in \eqref{w1w2w3} and \eqref{w3w4}.
    Now, we derive the desired bounds according to the value of $\mu$.

    \textbf{Case 1:} $\mu \in (0,1)$. Then we have $H_t = \widetilde{H}_0 t^{1-\mu}$ for all $t\ge 1$ due to \eqref{definition_Ht} and \eqref{w4w0H0}.  Therefore, we can deduce from \eqref{Gamma_t} that for all $t\ge 2$,
    \begin{align}
        &f(x^t)+g(y^t)+\beta_{t-1}h(x^t,y^t)-f(x^*)-g(y^*) \nonumber \\
        &\leq  \frac{\omega_1}{t(t\!+\!1)}\!+\!\frac{\omega_2}{(t\!+\!1)^{1-\delta}}
          \!+\!\frac{\omega_3}{(t\!+\!1)^{\nu}}\!+\!2D_f^2\frac{\widetilde{H}_0t^{1-\mu}}{t\!+\!1}\!+\!\frac{1}{t(t\!+\!1)} \sum\limits_{k=1}^{t-1}(k\!+\!1)(k\!+\!2)\zeta_k. \label{f+g+h-f*-g*_holder}
    \end{align}

    Next, we consider the term $ \sum\limits_{k=1}^{t-1}(k+1)(k+2)\zeta_k $. Let ${\frak T} = \{k:\; 1\le k\le t-1,\, \zeta_k > 0\}$. Recall that $\zeta_k = -\frac{H_k}{2}\|x^{k+1}-x^k\rVert^2 +\frac{M_f}{\mu+1}\|x^{k+1}-x^k\rVert^{\mu+1}$ (see \eqref{Definition_zeta}). Now, for $k\in {\frak T}$, $\zeta_k > 0$ and hence $\|x^{k+1} -x^k\|\neq 0$. Solving  $\zeta_k > 0$ for $\|x^{k+1} -x^k\|$, we deduce that
    \begin{align}\label{normx}
        \|x^{k+1}-x^k\rVert \leq \frac{1}{k}\left(\frac{2M_f}{(1+\mu)\widetilde{H}_0}\right)^{1/(1-\mu)}\ \ \ \ \ \ \forall k\in {\frak T}.
    \end{align}
    Then we have for all $t\ge 2$ that
    \begin{align}\label{sum_zeta}
      &\frac{1}{t(t+1)}\sum\limits_{k=1}^{t-1}(k+1)(k+2)\zeta_k \overset{(a)}\le \frac{1}{t(t+1)}\sum\limits_{k\in {\frak T}}(k+1)(k+2)\zeta_k   \nonumber \\
      \overset{(b)}=&   \frac{1}{t(t+1)}\sum\limits_{k\in {\frak T}}(k+1)(k+2)\left( -\frac{H_k}{2}\|x^{k+1}-x^k\rVert^2
      +\frac{M_f}{\mu+1}\|x^{k+1}-x^k\rVert^{\mu+1} \right)  \nonumber \\
\leq & \frac{1}{t(t+1)}\sum\limits_{k\in {\frak T}}(k+1)(2k+2)\frac{M_f}{\mu+1}\|x^{k+1}-x^k\rVert^{\mu+1} \nonumber \\
      \overset{(c)}\leq & \frac{1}{t(t+1)}\sum\limits_{k\in {\frak T}}(4k^2)\widetilde{H}_0\frac{2M_f}{(\mu+1)\widetilde{H}_0} \left(\frac{2M_f}{(1+\mu)\widetilde{H}_0}\right)^{\frac{1 +\mu}{1-\mu}} \frac{1}{k^{1+\mu}} \nonumber \\
      \overset{(d)}\leq & \frac{\omega_0}{t(t+1)}(t-1)t^{1-\mu}
 \leq  \frac{\omega_0}{t^{\mu}},
    \end{align}
    where (a) holds because $\zeta_k \le 0$ when $k\notin {\frak T}$, (b) follows from \eqref{Definition_zeta}, (c) follows from \eqref{normx} and the fact that $(k+1)^2 \le 4k^2$, and (d) holds thanks to the definition of $\omega_0$ in \eqref{w4w0H0} and the facts that $|{\frak T}|\le t-1$ and $k\mapsto k^{1-\mu}$ is increasing.
    Combining \eqref{sum_zeta} with \eqref{f+g+h-f*-g*_holder}, we obtain that
     \begin{align}
        &f(x^t)+g(y^t)+\beta_{t-1}h(x^t,y^t)-f(x^*)-g(y^*) \nonumber \\
        &\leq  \frac{\omega_1}{t(t+1)}+\frac{\omega_2}{(t+1)^{1-\delta}}
         +\frac{\omega_3}{(t+1)^{\nu}}+2D_f^2\widetilde{H}_0\frac{1}{(t+1)^{\mu}}+ \frac{\omega_0}{t^{\mu}},  \nonumber \\
        &=  \frac{\omega_1}{t(t+1)}+\frac{\omega_2}{(t+1)^{1-\delta}}
         +\frac{\omega_3}{(t+1)^{\nu}}+2D_f^2\widetilde{H}_0\frac{1}{(t+1)^{\mu}}+  \omega_0\left( \frac{t+1}{t} \right)^{\mu}\frac{1}{(t+1)^{\mu}}, \nonumber \\
        &\leq  \frac{\omega_1}{t(t+1)}+\frac{\omega_2}{(t+1)^{1-\delta}}
        +\frac{\omega_3}{(t+1)^{\nu}}+(2D_f^2\widetilde{H}_0+ 2\omega_0) \frac{1}{(t+1)^{\mu}}, \nonumber
    \end{align}
    where the last inequality holds because $\left( \frac{t+1}{t} \right)^{\mu} \leq 2^{\mu} \leq 2$. The desired conclusion follows upon recalling the definitions of  $\omega_4$ in  \eqref{w3w4}.

    \textbf{Case 2:} $\mu =1$.  By  \eqref{definition_Ht}, we have $H_t = \max\{H_0,M_f\} \geq M_f$ for all $t\ge 1$. Then, for all $t\ge 1$, we have
    \[
    \zeta_t =\! -\frac{H_t}{2}\|x^{t+1}-x^t\rVert^2 +\frac{M_f}{\mu+1}\|x^{t+1}-x^t\rVert^{\mu+1} \!\leq\!-\frac{M_f}{2}\|x^{t+1}-x^t\rVert^2 +\frac{M_f}{2}\|x^{t+1}-x^t\rVert^{2}\!=0.
    \]
    Using this together with \eqref{Gamma_t}, we have for all $t\ge 2$ that
     \begin{align}
        &f(x^t)+g(y^t)+\beta_{t-1}h(x^t, y^t)-f(x^*)-g(y^*) \nonumber \\
        &\leq  \frac{\omega_1}{t(t+1)}+\frac{\omega_2}{(t+1)^{1-\delta}} +\frac{\omega_3}{(t+1)^{\nu}} +\frac{\omega_5}{t+1}.\nonumber \label{conclusion 2}
    \end{align}
    The desired results hold upon recalling the definitions of  $\omega_5$ and $\widetilde H_0$ in  \eqref{w3w4} and \eqref{w4w0H0}.
\end{proof}
\end{color}

\bibliography{sample}

\end{document}